\documentclass[11pt,letterpaper]{amsart}
\usepackage[utf8]{inputenc}
\usepackage[margin=1in]{geometry}
\usepackage{amsmath}
\usepackage{amssymb}
\usepackage{amsthm}
\usepackage{amscd}
\usepackage{enumerate}
\usepackage{xcolor}
\usepackage{subcaption}

\usepackage{graphicx}
\usepackage{amsmath,amsthm,amsfonts,amscd,amssymb,comment,eucal,latexsym,mathrsfs}
\usepackage{stmaryrd}
\usepackage[all]{xy}

\usepackage{epsfig}

\usepackage[all]{xy}
\xyoption{poly}
\usepackage{fancyhdr}
\usepackage{wrapfig}
\usepackage{epsfig}

\theoremstyle{plain}
\newtheorem{thm}{Theorem}[section]
\newtheorem{prop}[thm]{Proposition}
\newtheorem{lem}[thm]{Lemma}
\newtheorem{cor}[thm]{Corollary}

\theoremstyle{definition}
\newtheorem{defn}[thm]{Definition}
\theoremstyle{remark}

\newtheorem{example}{Example}

  \def\C{{\mathbb{C}}}  \def\E{{\mathbb{E}}}        \def\M{{\mathbb{M}}} \def\N{{\mathbb{N}}}    \def\R{{\mathbb{R}}}        

     \def\cF{{\mathcal{F}}}  \def\cH{{\mathcal{H}}}   \def\cK{{\mathcal{K}}}       \def\cR{{\mathcal{R}}} \def\cS{{\mathcal{S}}}  \def\cU{{\mathcal{U}}}     

\newcommand\Aut{\operatorname{Aut}}

\newcommand\Ball{{\operatorname{Ball}}}

\newcommand\co{\operatorname{co}}

\newcommand\dom{\operatorname{dom}}

\newcommand\esssup{\operatorname{esssup}}

\newcommand\Fix{{\operatorname{Fix}}}

\newcommand\id{\operatorname{id}}

\newcommand\Meas{{\operatorname{Meas}}}

\newcommand\Prob{\operatorname{Prob}}
\newcommand\ran{\operatorname{ran}}

\newcommand\supp{\operatorname{supp}}

\newcommand\Stab{\operatorname{Stab}}

\newcommand\Span{\operatorname{span}}

\newcommand\gr{\operatorname{gr}}

\newcommand\tr{\operatorname{tr}}

\newcommand{\actson}{\curvearrowright}

\newcommand{\acston}{\actson}

\newcommand{\ip}[1]{\langle #1 \rangle}

\begin{document}
\title{Coamenability and strong ergodicity}

\author{Ben Hayes}
\address{Department of Mathematics, University of Virginia\\
141 Cabell Drive, Kerchof Hall
P.O. Box 400137,
Charlottesville, VA 22904}
\email{brh5c@virginia.edu}

\begin{abstract}
 Following methods of Bannon-Marrakchi-Ozawa, we show that for coamenable inclusion $\cS\leq \cR$ of ergodic, probability measure-preserving relations, we have that $\cR$ is strongly ergodic if and only if $\cS$ is strongly ergodic. More general results are given when $\cS\leq \cR$ is coamenable, $\cR$ is strongly ergodic, but we do not assume ergodicity of $\cS$. As a consequence, if $\Lambda\leq \Gamma$ is a coamenable inclusion of groups, then any strongly ergodic $\Gamma$ action has countably many ergodic components for the $\Lambda$ action, each of which is strongly ergodic.
\end{abstract}

\maketitle

\section{Introduction}

Strong ergodicity is a concept that arose out of the works \cite{CWPropT, SchmidtCohom, SchmidtSpectralGap}  and is a fundamental topic in the study of ergodic theory of nonamenable groups and measured group theory. If $\Gamma$ is a countable group, and $(X,\mu)$ is a standard probability space, we say that $\Gamma\actson (X,\mu)$ is \emph{strongly ergodic} if whenever $(E_{n})_{n}$ is a sequence of measurable subsets of $X$ with $\mu(g E_{n}\Delta E_{n})\to_{n\to\infty}0$ for all $g\in \Gamma,$ then we necessarily have that $\mu(E_{n})(1-\mu(E_{n}))\to 0$ (so up to subsequences, either $\mu(E_{n})\to 0$ or $\mu(E_{n})\to 1$). It is direct to see that if $(E_{n})_{n}$ is a sequence of measurable subsets of $X$ with $\mu(g E_{n}\Delta E_{n})\to_{n\to\infty}0$ for all $g\in \Gamma,$ then $\mu(\gamma E_{n}\Delta E_{n})\to 0$ for all $\gamma$ in the full group, $[\cR_{\Gamma,X}]$ of $\Gamma\actson (X,\mu)$ defined by
\[[\cR_{\Gamma,X}]=\{\gamma\colon X\to X: \text{$\gamma$ is a bimeasurable bijection and } \gamma(x)\in \Gamma x \text{ for all almost every $x\in X$}\}.\]
It can thus be regarded as  property of the orbit equivalence relation of this action, $\cR_{\Gamma,X}=\{(x,\gamma(x)):X\in X\}.$ More generally, there is a general notion of strong ergodicity for discrete, probability measure-preserving, orbit equivalence relations (we remark that there are genuine subtleties in considering strong ergodicity for actions when dropping the measure-preserving assumption, see e.g the recent paper \cite{giritlioglu2026translationactionsnonunimodulargroups}).  In particular, strong ergodicity of the action is an orbit equivalence invariant for the action. 

Part of its importance is its connection to many notions and problems in the field. Strong ergodicity is implied by spectral gap  (i.e. absence of almost-invariant unit vectors) of the Koopman representation of $\Gamma$ on $L^{2}$-functions with mean zero (the converse is false, see \cite[Example 2.7]{SchmidtSpectralGap}) and is equivalent to spectral gap of the Koopman representation of the full group of $\Gamma\actson (X,\mu)$ (see \cite[Theorem 2.4]{SchmidtCohom}). It can thus be viewed as both the equivalence relation version of spectral gap, and as such asserts some version of \emph{expansion} (in the sense of expander graphs) for the relation. It also admits a cohomological characterization see \cite{SchmidtCohom}, and can be characterized (in the ergodic case) as not admitting the hyperfinite $\textrm{II}_{1}$-relation as a quotient \cite[Theorem 2.2]{JonesSchmidt}. As such, it has connections with the notion of stable relations \cite[Section 3]{JonesSchmidt}.

There is a fruitful analogy between strongly ergodic relations and full von Neumann algebras, in the sense of Connes \cite{ConnesSE, ConnesAP}. Indeed, it is direct to see from the definitions that if the von Neumann algebra of an orbit equivalence relation is full, then that relation is strongly ergodic. The converse is false (see \cite{ConnesJonesCartan} for a counterexample). 

There is a natural notion of coamenability for inclusions of discrete, probability measure-preserving equivalence relations, e.g. one could simply demand that the corresponding inclusion of von Neumann algebras is a relatively amenable inclusion in the sense of \cite{PopaCorr}.
In \cite{HayesCoAmen}, we investigated several other equivalent ways of defining an inclusion of relations to be coamenable, including several equivalences which are analogous to characterizations of amenability of relations given in \cite{CFW, OrnWeiss, ZimmerAMen1, ZimmerAmen2}.

In \cite{BMOFull}, Bannon-Marrakchi-Ozawa showed that if $N\leq M$ is an inclusion of von Neumann algebras which are factors (i.e. with trivial center), and if $M$ is full, then so is $N$ (in fact, the authors have general results only assuming fullness and factoriality of $M$). This resolved a question of Popa from \cite{PopaCorr}. Due to the analogy between fullness and strong ergodicity, at the CIRM conference ```Orbit equivalence and topological and measurable dynamics" Marrakchi asked us the following: if $\cS\leq \cR$ is a coamenable inclusion of ergodic, probability measure-preserving orbit equivalence relations, is it true that $\cS$ is strongly ergodic if and only if $\cR$ is strongly ergodic? In this paper, we showed that the answer to Marrakchi's question is yes.

\begin{thm}\label{thm: intro thm}
Let $\cS\leq \cR$ be ergodic, discrete, measure-preserving equivalence relations on a standard probability space $(X,\mu)$. If $\cS\leq \cR$ is coamenable, then $\cR$ is strongly ergodic if and only if $\cS$ is strongly ergodic.     
\end{thm}

By the ergodic decomposition, the study of inclusions $\cS\leq \cR$ of probability measure-preserving equivalence relations can be reduced to the case where $\cR$ is ergodic. However, we cannot further reduce to the case that $\cS$ is also ergodic. Thus, it is natural to investigate a generalization of the above Theorem to the case where $\cR$ is assumed strongly ergodic and where we only assume $\cS$ is coamenable in $\cR$. 

\begin{thm}\label{thm: I said the real intro thm}
Let $\cS\leq \cR$ be ergodic, discrete, measure-preserving equivalence relations on a standard probability space $(X,\mu)$. If $\cS\leq \cR$ is   coamenable, and $\cR$ is strongly ergodic, then there is a positive measure, $\cS$-invariant set $E\subseteq X$ so that $\cS|_{E}$ is strongly ergodic.
\end{thm}

Recall that a probability measure-preserving relation $\cR$ on a standard probability space $(X,\mu)$ has \emph{countable ergodic decomposition} if there is a countable set $I$, and a  partition (modulo null sets) $(E_{i})_{i\in I}$ by positive measure, $\cR$-invariant subsets of $X$ so that $\cR|_{E_{i}}$ is ergodic for each $i\in I$. We refer to $\cR|_{E_{i}}$ as the \emph{ergodic components} of $\cR$. 

It would be natural to iterate Theorem \ref{thm: I said the real intro thm} via a measure exhaustion argument to deduce that a coamenable subrelation of a strongly ergodic relation has countably many ergodic components, each of which are strongly ergodic. To do this, one would simply have to prove that any positive measure subset $F\subseteq X$ contains a positive measure set $F_{0}$ for which $\cS|_{F_{0}}$ is strongly ergodic. However, a subtlety appears when attempting this argument which is that (see \cite[Example 1]{HayesCoAmen}) there are cases where $\cS\leq \cR$ is coamenable, but nevertheless there are positive measure subsets $E$ of $X$ for which $\cS|_{E}\leq \cR|_{E}$ is not coamenable. The relevant notion that fixes this is to assume that $\cS\leq \cR$ is \emph{everywhere coamenable}, meaning that $\cS|_{E}\leq \cR|_{E}$ is coamenable for every positive measure $E\subseteq X$. We refer to \cite[Section 4]{HayesCoAmen} for a detailed discussion of other equivalent notions. This is precisely the condition we need to apply a measure exhaustion argument, and Theorem \ref{thm: I said the real intro thm} implies the following corollary. 

\begin{cor}\label{cor: general SE statmeent intro}
Let $\cS\leq \cR$ be ergodic, discrete, measure-preserving equivalence relations on a standard probability space $(X,\mu)$. If $\cS\leq \cR$ is  everywhere coamenable, and $\cR$ is strongly ergodic, then $\cS$ has countable ergodic decomposition and each ergodic component of $\cS$ is strongly ergodic.
\end{cor}

We remark that the conclusion of the above theorem is related to work of Chifan-Ioana \cite{SolidErg}, indeed they proved the remarkable result that if $\cS$ is a nowhere amenable relation (in the sense that all nontrivial restrictions are nonamenable) which is a subquotient relation of a Bernoulli orbit equivalence relation of a nonamenable group, then $\cS$ has countable ergodic decomposition and each ergodic component of $\cS$ is strongly ergodic. As discussed in \cite{SolidErg}, this property is of utility in the study of percolations of graphs (see \cite{csoka2025quantitativeindistinguishabilitysparsedense} for recent connections) and can be used to simplify part of the celebrated Gaboriau-Lyons theorem \cite{Gaboriau-Lyons} asserting that every nonamenable group contains a ``measurable" copy of the free group on $2$-generators.

In \cite[Proposition 7.2]{HayesCoAmen} we show that if $\Lambda\leq \Gamma$ is a coamenable inclusion of countable groups, and if $\Gamma\actson (X,\mu)$ is any probability measure-preserving action (not assumed essentially free) on a standard probability space $(X,\mu)$, then the orbit equivalence relation for the $\Lambda$ action is everywhere coamenable in the orbit equivalence relation for the $\Gamma$ action. Thus Corollary \ref{cor: general SE statmeent intro} immediately implies the following.

\begin{cor}\label{cor: intro groups}
Let $\Lambda\leq \Gamma$ be a coamenable inclusion of countable groups. Suppose $(X,\mu)$ is a standard probability space and $\Gamma\actson(X,\mu)$ is a strongly ergodic probability measure-preserving action. Then $\Lambda\acston (X,\mu)$ has countably many ergodic components, each of which is strongly ergodic. 
\end{cor}

\subsection{Connections to $C^{*}$-Algebra Theory.} Our methods in this paper rely heavily on the representation theory of $\cR$. 
 We recall the precise notion of a representation in Section \ref{sec: unitary rep}, but the essential idea is that one should have a unitary representation of the full group, together with an appropriately continuous representation of $L^{\infty}(X)$, and these representations should be compatible in a natural sense. 
 In Section \ref{sec: unitary rep}, we rephrase this notion as being equivalent to a $*$-representation of a natural $*$-algebra (known as the groupoid ring $\C(\cR)$) which contains $L^{\infty}(X,\mu)$ as a $*$-subalgebra and $[\cR]$ as a subgroup of the unitary group (see Definition \ref{defn: groupoid ring} for more details). We use the fact that unitary representations play an important role in our proof as an opportunity to define a natural $C^{*}$-completion, denoted $C^{*}(\cR)$, of $\C(\cR)$. representations  of $\cR$.
We connect weak containment of representations of $\cR$ to weak containment of this $C^{*}$-algebra. One unitary representation of $\cR$ that is of interesting is the representation $\alpha$ of $\cR$ on $L^{2}(X)$ where $L^{\infty}(X)$ acts by multiplication operators and $[\cR]$  acts by $\alpha_{\gamma}(f)(x)=f(\gamma^{-1}(x))$. As we discuss in \ref{sec: unitary rep}, this example should be thought of the ``trivial" representation of $\cR$. We characterize strong ergodicity as saying that any representation of $\cR$ is weakly equivalent to $L^{2}(X)$ must contain $L^{2}(X)$ (analogous to \cite[Proposition 3.2]{BMOFull} ), and characterize almost invariant vectors via weak containment of the representation of $\cR$ on $L^{2}(X)$. This provides further justification that $L^{2}(X)$ should be thought of as the ``trivial representation" of $\cR$. We further characterize coamenability of inclusions view weak containment of $L^{2}(X)$ into $L^{2}(\cR/\cS)$.

\subsection{Strategy of the proof.}

Our main strategy follows that of \cite{BMOFull}. We sketch the proof in the case that both $\cS,\cR$ are ergodic. 
 First, as mentioned above, we show
(see Corollary \ref{cor: weak containment strong ergodic}) that strong ergodicity of $\cR$ is equivalent to saying that if a representation $\cH$ is weakly equivalent to $L^{2}(X)$ (in the sense of weak containment, see the discussion in Section \ref{sec: weak containment}), then $L^{2}(X)$ embeds in $\cH$. This is similar to \cite[Proposition 3.2]{BMOFull}.

We also use that $\cS$ being strongly ergodic is the same as the action of $[\cS]$ on the ultrapower algebra $L^{\infty}(X)^{\omega}$   (discussed in Section \ref{sec: ultrapower}
) having no nonconstant invariant vectors. Adapting \cite[Proof of Theorem A]{BMOFull}, we build an intermediate relation $\cS\leq \widehat{\cS}\leq \cR$ with the property that if $\gamma\in [\cR]$ acts trivially on all elements of $L^{\infty}(X)^{\omega}$ which are fixed by $\widehat{\cS}$ then $\gamma\in [\widehat{\cS}]$. Adapting \cite[Lemma 3.4]{BMOFull}, we show that this property implies that $L^{2}(\cR/\widehat{\cS})$ is weakly contained in $L^{2}(X)$. Coamenability of $\cS\leq \cR$ then forces $L^{2}(X)$ to be weakly equivalent to $L^{2}(\cR/\widehat{\cS})$. As mentioned above, strong ergodicity of $\cR$ then implies that $L^{2}(\cR/\cS)$ has an invariant vector, and analogous to the group setting (see Corollary \ref{cor: weak containment strong ergodic}) we must have that $\widehat{\cS}$ is finite index in $\cR$. This implies that $\widehat{\cS}$ is strongly ergodic, and as in \cite[Proof of Theorem A]{BMOFull}, the way $\widehat{\cS}$ is constructed will then imply strong ergodicity of $\cS$. Minor changes have to be made here in the case that $\cS$ is not ergodic, namely we end up only having that $\widehat{\cS}$ is finite index in $\cR$ after restricting to a positive measure $\cS$-invariant set. Ultimately this leads to the proof of \ref{thm: I said the real intro thm}, which as mentioned implies Corollaries \ref{cor: general SE statmeent intro},\ref{cor: intro groups}.

\textbf{Acknowledgements.} I thank Amine Marrakchi for asking me at the CIRM conference ``Orbit equivalence and topological and measurable dynamics" whether for a coamenable inclusion $\cS\leq \cR$ we have that $\cS$ is strongly ergodic if and only if $\cR$ is, which was the inception for this work. I thank CIRM for its hospitality and the organizers of this conference for hosting me. I thank Felipe Flores for interesting discussions related to this work. 

\tableofcontents

\section{Preliminaries}

A \emph{standard measure space} is a  pair $(X,\mu)$ where $X$ is a standard Borel space, and $\mu$ is a $\sigma$-finite Borel measure on $X$. It is called a \emph{standard probability space} if $\mu$ is a probability measure. 
We slightly abuse notation and say that $E\subseteq X$ is \emph{$\mu$-measurable} if it is in the domain of the completion of $\mu$. If $\mu$ is clear from the context, we will often simply say ``measurable".
An \emph{equivalence relation over $(X,\mu)$} is a Borel subset $\cR\subseteq X\times X$ so that the relation $\thicksim$ on $X$ given by $x\thicksim y$ if $(x,y)\in \cR$ is an equivalence relation. For $x\in X$, we let $[x]_{\cR}=\{y\in X:(x,y)\in \cR$. We say that $\cR$ is \emph{discrete} if for almost every $x\in X$ we have that $[x]_{\cR}$ is countable. If $\cR$ is discrete, we may turn $\cR$ into a $\sigma$-finite measure space by endowing $\cR$ with the Borel measure
\[\mu_{\cR}(E)=\int_{X}|\{y:(x,y)\in E\}|\,d\mu(x)\mbox{ for all Borel $E\subseteq \cR$.}\]
We will continue to use $\mu_{\cR}$ for the completion of $\mu_{\cR}$. If $\cR$ is discrete, we say that it is \emph{measure-preserving} if the map $\cR\to \cR$ given by $(x,y)\mapsto (y,x)$ preserves $\mu_{\cR}$. We say that $\cR$ is \emph{probability measure-preserving} if $\cR$ is  measure-preserving and $\mu$ is a probability measure.  If $(X,\mu)$ is a probability space, and $E\subseteq X$ is Borel with $\mu(E)>0$ we let
\[\cR|_{E}=\cR\cap (E\times E).\]
If $E$ has positive measure, then $\cR|_{E}$ is a measure-preserving relation over the probability space $(E,\frac{\mu(E\cap \cdot)}{\mu(E)})$. If $\Gamma$ is a countable group, and $\Gamma\actson (X,\mu)$ is a measure-preserving action, with $\Gamma\actson X$ Borel, then $\cR_{\Gamma,X}=\{(x,gx):g\in \Gamma\}$ is a discrete, measure-preserving equivalence relation. We let $[\cR]$ be  all bimeasurable bijections $\gamma\colon X\to X$ so that $\gamma(x)\in [x]_{\cR}$ for almost every $x\in X$. We identify two elements of $[\cR]$ if they agree almost everywhere.  We let $[[\cR]]$ be all bimeasurable bijections $\gamma\colon A\to B$ such that:
\begin{itemize}
\item $A,B\subseteq X$ are measurable,
\item $\gamma(x)\in [x]_{\cR}$ for almost every $x\in A$.
\end{itemize}
We usually use $\dom(\gamma),\ran(\gamma)$ for $A,B$. We typically identify two elements of $[[\cR]]$ if they agree up to sets of measure zero. However just as with the theory of $L^{p}$-spaces we will often blur the lines between a bimeasurable bijection as above and its equivalence class in $[[\cR]]$. Thus, for example, we will say that $\gamma\in [[\cR]]$ is \emph{Borel} to mean bimeasurable bijection as above where $A,B$ are Borel subsets of $X$ and $\gamma$ is a Borel map. 
For $\gamma\in [[\cR]]$, we let $\gamma^{-1}$ be the compositional inverse of $\gamma$. 
If $\gamma_{1},\gamma_{2}\in [[\cR]]$ we define $\gamma_{1}\circ \gamma_{2}$ by saying that $\dom(\gamma_{1}\circ \gamma_{2})=\dom(\gamma_{2})\cap \gamma_{2}^{-1}(\dom(\gamma_{1}))$, and $(\gamma_{1}\circ \gamma_{2})(x)=\gamma_{1}(\gamma_{2}(x))$. 
For $\gamma\in [[\cR]]$ and $f\colon X\to \C$, we define $\alpha_{\gamma}(f)\colon X\to \C$ by $[\alpha_{\gamma}(f)](x)=1_{\ran(\gamma)}(x)f(\gamma^{-1}(x))$. This defines a bounded operator on $L^{2}(X)$, which we still denote as $\alpha_{\gamma}$. 

For Hilbert spaces $\cH,\cK$ we recall the strong and weak operator topologies on $B(\cH,\cK)$. The \emph{strong operator topology} is defined by saying that for a $T\in B(\cH,\cK)$ we have a basis of open neighborhoods of $T$ given 
\[U_{F,\varepsilon}(T)=\bigcap_{\xi\in F}\{S\in B(\cH,\cK):\|(T-S)\xi\|<\varepsilon\},\]
where $F,\varepsilon$ range over finite $F\subseteq \cH$ and $\varepsilon>0$. Similarly, the \emph{weak operator topology} is given by saying that for $T\in B(\cH,\cK)$ we have a basis of open neighborhoods of $T$ given by
\[V_{F,G,\varepsilon}(T)=\bigcap_{\xi\in F,\eta\in G}\{S\in B(\cH,\cK):|\ip{(T-S)\xi,\eta}|<\varepsilon\},\]
where $F\subseteq \cH,G\subseteq \cK$ are finite and $\varepsilon>0$. We often use SOT and WOT to abbreviate strong operator topology and weak operator topology.

If $V$ is a vector space over the reals (or complexes) and $E\subseteq V$, we use $\co(E)$ for its convex hull,
\[\co(E)=\bigcup_{n=1}^{\infty}\left\{\sum_{j=1}^{n}\lambda_{j}x_{j}:x_{1},\cdots,x_{n}\in E, \lambda_{1},\cdots,\lambda_{n}\in [0,1],\sum_{j=1}^{n}\lambda_{j}=1\right\}.\]
If $A$ is a unital $C^{*}$-algebra a \emph{state} is, by definition, a $\varphi\in A^{*}$ so that $\varphi(1)=1$ and $\varphi(x^{*}x)\geq 0$ for all $x\in A$. 

For a normed vector space $V$, we use $\Ball(V)=\{v\in V:\|v\|\leq 1\}$. We recall that the \emph{weak topology} on $V$ is defined by saying that a neighborhood basis of $v\in V$ is given by
\[W_{F,\varepsilon}(v)=\bigcap_{L\in F}\{w\in V:|L(v-w_|<\varepsilon\},\]
where $F$ ranges over finite subsets of $V^{*}$ and $\varepsilon$ ranges over all positive numbers. For $A\subseteq V$, we use $\overline{A}^{wk}$ for its weak closure. By contrast, we recall the \emph{weak$^{*}$ topology} on $V^{*}$ is defined by saying that a neighborhood basis of $L\in V^{*}$ is given by
\[\widetilde{W}_{F,\varepsilon}(L)=\bigcap_{v\in F}\{\widetilde{L}\in V^{*}:|(L-\widetilde{L})(v)|<\varepsilon\},\]
where $F$ ranges over finite subsets of $V$ and $\varepsilon$ ranges over all positive real numbers. For $B\subseteq V^{*}$ we use $\overline{B}^{wk^{*}}$ for its weak$^{*}$ closure.

\section{Unitary representations of orbit equivalence relations}\label{sec: unitary rep}

We begin by recalling the general notion of a measurable field of Hilbert spaces over a probability space, and what a unitary representation of an orbit equivalence relation means in this context.

\begin{defn}Let  $(X,\mu)$ be a standard probability space, then a \emph{measurable field of Hilbert spaces over $X$} is a family $(\mathcal{H}_{x})_{x\in X}$ of separable Hilbert spaces, together with a family $\Meas(\mathcal{H}_{x})\subseteq \prod_{x\in X}\mathcal{H}_{x}$  so that:
\begin{itemize}
\item for every $(\xi_{x})_{x},(\eta_{x})_{x}\in \Meas(\mathcal{H}_{x})$ we have that $x\mapsto \ip{\xi_{x},\eta_{x}}$ is measurable,
\item if $\eta=(\eta_{x})_{x\in X}\in \prod_{x\in X}\mathcal{H}_{x}$ and $x\mapsto \ip{\xi_{x},\eta_{x}}$ is measurable for all $\xi=(\xi_{x})_{x}\in \Meas(\mathcal{H}_{x}),$ then $\eta\in \Meas(\mathcal{H}_{x}),$
\item there is a sequence $(\xi^{(n)})_{n=1}^{\infty}$ with $\xi^{(n)}=(\xi^{(n)})_{x\in X}$ in $\Meas(\mathcal{H}_{x})$ so that $\mathcal{H}_{x}=\overline{\Span\{\xi^{(n)}_{x}:n\in \N\}}$ for almost every $x\in X.$
\end{itemize}
The direct integral, denoted $\int_{X}^{\oplus} \mathcal{H}_{x}\,d\mu(x),$ is defined to be all $\xi\in \Meas(\mathcal{H}_{x})$ so that $\int_{X}\|\xi_{x}\|^{2}\,d\mu(x)<\infty,$ where we identify two elements of $\Meas(\mathcal{H}_{x})$ if they agree outside a set of measure zero. We put an inner product on $\int_{X}^{\oplus}\mathcal{H}_{x}$ by
\[\ip{\xi,\eta}=\int_{X}\ip{\xi_{x},\eta_{x}}\,d\mu(x),\]
and this gives $\int_{X}^{\oplus}\mathcal{H}_{x}$ the structure of a Hilbert space.
\end{defn}
We shall typically drop ``over $X$" in ``a measurable field of Hilbert space over $X$" if $X$ is clear from the context.
Suppose that $(T_{y})_{y\in Y}\in \prod_{y\in Y}B(\cH_{y})$. We then say that $(T_{y})_{y\in Y}$ is a \emph{measurable field of operators} if for any $(\xi_{y})_{y\in Y}\in \Meas(\mathcal{H}_{y})$, we have that $(T_{y}\xi_{y})_{y\in Y}\in \Meas(\mathcal{H}_{y})$. Suppose that $(T_{y})_{y\in Y}$ is a measurable field of operators and $\esssup_{y\in Y}\|T_{y}\|<+\infty$. We may then define $\int_{Y}^{\bigoplus}T_{y}\,d\nu(y):=T$ by $(T\xi)_{y}=T_{y}\xi_{y}$. By \cite[Section IV.7 Equation (6)]{Taka}, it follows that 
\[\left\|\int_{Y}^{\bigoplus}T_{y}\,d\nu(y)\right\|=\esssup_{y\in Y}\|T_{y}\|.\]

\begin{defn}\label{defn:rep of a relation}
Let $\cR$ be a discrete, measure-preserving equivalence relation on $(X,\mu)$. A representation of $\cR$ is a measurable field $\mathcal{H}_{x}$ of Hilbert spaces over $(X,\mu),$ as well as unitaries $\pi(x,y)\in \mathcal{U}(\mathcal{H}_{y},\mathcal{H}_{x})$ satisfying the following axioms:
\begin{itemize}
\item For every pair  measurable sections $(\xi_{x})_{x},(\eta_{x})_{x}$ of $(\mathcal{H}_{x})_{x}$ the function $(x,y)\in \cR\mapsto \ip{\pi(x,y)\xi_{y},\eta_{x}}$ is measurable.
\item For almost every $x\in X,$ it is true that for all $y,z\in [x]_{\cR}$ we have $\pi(x,y)\pi(y,z)=\pi(x,z).$
\end{itemize}
\end{defn}

We will typically say ``let $((\cH_{x})_{x\in X},\pi)$ be a unitary representation of $\cR$" to indicate that $(\cH_{x})_{x\in X}$ is a measurable field of Hilbert spaces, and $\pi(x,y)\in \cU(\mathcal{H}_{y},\mathcal{H}_{x})$ satisfies the above.

Let us present a few examples of representations of $[\cR]$.

\begin{example}\label{ex: group case}
Suppose $\Gamma$ is a countable group and $\Gamma\actson (X,\mu)$ is a free, probability measure-preserving action. Let $\pi\colon G\to \cU(\cH)$ be a unitary representation. In this case, we let $\cR=\{(x,gx):x\in X,g\in \Gamma\}$. We then set $\cH_{x}=\cH$ and $\pi(x,gx)=g^{-1}$ for all $g\in G$ and almost every $x\in X$ (this is almost surely well-defined since the action is free).
\end{example}

\begin{example}\label{ex: trival rep} Consider the representation $\alpha\colon [\cR]\to \mathcal{U}(L^{2}(X,\mu))$ given by $(\alpha(\gamma)f)(x)=f(\gamma^{-1}(x))$ for all $f\in L^{2}(X,\mu),x\in X,\gamma\in [\cR]$. The space $L^{2}(X,\mu)$ can be written as the direct integral of $\mathcal{H}_{x}=\C$ and $\alpha$ is induced by the maps $\alpha(x,y)\in \mathcal{U}(\C,\C)$ for $(x,y)\in \cR$ by $\alpha(x,y)=\id$. This example should be regarded as analogous to the trivial representation of a group. Indeed, if $\cR$ is given by a free action of a group, then this example is isomorphic with Example \ref{ex: group case} with $\pi$ being the trivial representation on a one-dimensional Hilbert space. 
\end{example}

\begin{example}\label{ex: left/right reg rep}
We can write $L^{2}(\cR)$ as the direct integral of $\ell^{2}([x]_{\cR})$ by identifying $f\in L^{2}(\cR)$ with $f=(f_{x})_{x}\in \Meas(\ell^{2}([x]_{\cR}))$ given by $f_{x}(y)=f(x,y)$ and $\lambda(x,y)\in \mathcal{U}(\ell^{2}([x]_{\cR}),\ell^{2}([y]_{\cR}))=\mathcal{U}(\ell^{2}([x]_{\cR})$ is  given by $\lambda(x,y)=\id$.
We can also consider $L^{2}(\cR)$ as the direct integral of $\ell^{2}([y]_{\cR})$ in a second manner, by identifying $f\in L^{2}(\cR)$ with $f=(f^{y})_{y}\in\Meas(\ell^{2}([y]_{\R}))$ by $f^{y}(x)=f(x,y)$. Under this identification, we define $r(x,y)\in \mathcal{U}(\ell^{2}([x]_{\cR}),\ell^{2}([y]_{\cR}))=\mathcal{U}(\ell^{2}([x]_{\cR}))$ by $r(x,y)=\id$. 

We may regard $\lambda,r$ as the analogues of the left/right regular representations of a group. Indeed, if $\cR$ is given by a free action of a group $\Gamma$ then $\lambda,r$ are isomorphic to the example given in \ref{ex: group case} with $\pi$ being the left/right regular representation.
\end{example}

\begin{example}\label{ex: quasi-regular rep}
Suppose $\cS\leq \cR$ are discrete, measure-preserving equivalence relations over a standard probability space $(X,\mu)$. Since $\cS$ is a subrelation of $\cR$, each $\cR$-equivalence class is a disjoint union of $\cS$-equivalence classes. For $x\in X$, we let $[x]_{\cR}/\cS$ be the space of $\cS$-equivalence classes in $[x]_{\cR}=\{y\in X:(x,y)\in \cR\}$.

We can give the family $(\ell^{2}([x]_{\cR}/\cS))_{x}$ a measurable structure by declaring that $\xi=(\xi_{x})_{x}\in \prod_{x\in X}\ell^{2}([x]_{\cR}/\cS)$ is measurable if $x \mapsto \xi_{x}([\gamma(x)]_{\cS})$ is measurable, for all $\gamma\in [\cR]$. General facts about direct integrals imply this collection of measurable vectors satisfy the above axioms (see Lemma 3.3 in \cite{AFH}). 

It is direct to check that $\int_{X}^{\oplus}\ell^{2}([x]_{\cR}/\cS)\,d\mu(x)$ is canonically isomorphic with the space $L^{2}(\cR/\cS)$ discussed in \cite{HayesCoAmen, AFH} and that the representation $\lambda_{\cR/\cS}$ discussed therein corresponds to the unitary representation $\lambda_{\cR/\cS}(x,y)=\id.$

The representation $\lambda_{\cR/\cS}$ should be regarded as analogous to the quasi-regular representation of a group $\Gamma$ on $\ell^{2}(\Gamma/\Delta)$ for a subgroup $\Delta$. If $\cR$ is a given by a free action of a countable group $\Gamma$ and $\cS$ is given by the action of a subgroup $\Lambda$, then $\lambda_{\cR/\cS}$ is isomorphic to the example given in Example \ref{ex: group case} with $\pi$ being the quasi-regular representation on $\ell^{2}(\Gamma/\Delta)$. 
\end{example}

Given a representation $(\mathcal{H}_{x},\pi)$ of $\cR,$ let $\mathcal{H}=\int_{X}^{\bigoplus}\mathcal{H}_{x}\,d\mu(x).$ We will proceed to discuss how the representation $\pi$ allows us to define bounded operators on this total space $\cH$.

\begin{defn}\label{defn: groupoid ring}
Let $\cR$ be a discrete, probability measure-preserving equivalence relation on a standard probability space $(X,\mu)$. We define the groupoid ring $\C(\cR)$ to be all $k\in L^{\infty}(\cR)$ so that the functions 
\[y\mapsto |\{(x,y):k(x,y)\ne 0\}|,\]
\[x\mapsto |\{(x,y):k(y,x)\ne 0\}\,\]
are essentially bounded functions on $X$. We turn $\C(\cR)$ into an algebra using the convolution product 
\[(k_{1}*k_{2})(x,y)=\sum_{z\in [x]_{\cR}}k_{1}(x,z)k_{2}(z,y)\]
for multiplication. Equipping $\C(\cR)$ with the $*$-operation
\[k^{*}(x,y)=\overline{k(y,x)},\]
we see that $\C(\cR)$ is a $*$-algebra. 
\end{defn}
We regard $L^{\infty}(X)\subseteq \C(\cR)$ via $f\mapsto ((x,y)\mapsto \delta_{x=y}f(x)$). It is an exercise to check that this embedding is an injective $*$-homomorphism. We have an injective homomorphism $[\cR]\subseteq \cU(\C(\cR)):=\{u\in \C(\cR):u^{*}u=1=uu^{*}\}$ via $\gamma\mapsto 1_{\gr(\gamma^{-1})}$, where $\gr(\gamma^{-1})=\{(x,\gamma^{-1}(x)):x\in X\}.$ We typically identify $[\cR]$ with its image under this homomorphism. 
By \cite[Lemma 3.3]{SauerBetti}, it follows that for every $f\in \C(\cR)$ we may write $f=\sum_{\gamma\in F}f_{\gamma}\gamma$ where $f_{\gamma}\in L^{\infty}(X)$ and $F\subseteq[\cR]$ is finite. It is also direct to check that 
\[\gamma f\gamma^{-1}=\alpha_{\gamma}(f) \text{ for all $\gamma\in [\cR],f\in L^{\infty}(X)$.}\]

Suppose $((\cH)_{x\in X},\pi)$ is a unitary representation of $\cR$, and $f\in \C(\cR)$, then for $\xi=(\xi_{x})_{x}\in \cH:=\int_{X}^{\oplus} \cH_{x}\,d\mu(x)$ we define
\[(\pi(f)\xi)_{x}=\sum_{y}f(x,y)\pi(x,y)\xi_{y}.\]
Let us show this is a bounded operator. Choose $C\geq 0$ so that 
\[\sum_{y}|f(x,y)|\leq C,\]
\[\sum_{y}|f(y,x)|\leq C\]
for almost every $x\in X$. 
Note that 
\begin{align*}
 \|\pi(f)\xi\|_{\cH}^{2}\leq \int_{X}\left(\sum_{y}|f(x,y)|\|\xi_{y}\|\right)^{2}\,d\mu(x)&\leq \int_{X} \left(\sum_{y}|f(x,y)|\right)\left(\sum_{y}|f(x,y)|\|\xi_{y}\|^{2}\right)\,d\mu(x)\\
 &\leq C\int_{X} \sum_{y}|f(x,y)|\|\xi_{y}\|^{2}\,d\mu(x)\\
&=C\int_{X}\sum_{y}|f(y,x)|\|\xi_{x}\|^{2}\,d\mu(x)\\
 &\leq C^{2}\int_{X}\|\xi_{x}\|^{2}\,d\mu(x)\\
 &=C^{2}\|\xi\|_{\cH}^{2}.
\end{align*}
We remark that the above argument is similar to what occurs in \cite[Proposition 4.2]{ElekLip}.
It is direct to check that the map $\pi\colon \C(\cR)\to B(\cH)$ defined above is $*$-homomorphism.
 Under the identification discussed above, we have nice formulas for $\pi$ on $L^{\infty}(X,\mu)$, $[\cR]$:
\[\pi(f)=\int_{X}^{\oplus}f(x)\,d\mu(x) \textnormal{ for all $f\in L^{\infty}(X,\mu)$,}\]
\[(\pi(\gamma)\xi)_{x}=\pi(x,\gamma^{-1}(x))\xi_{\gamma^{-1}(x)}, \textnormal{ for all $\gamma\in [\cR]$,$x\in X$, and $\xi=(\xi_{x})_{x}\in \int_{X}^{\oplus}\cH_{x}\,d\mu(x).$}\]
Because of the above discussion, if $(\Meas((\cH_{x})_{x\in X},\pi)$ is a unitary representation of $\cR$, and $\cH=\int_{X}^{\oplus}\cH_{x}\,d\mu(x)$ we will often refer to $\cH$ as a representation of $\cR$, instead of the pair $(\Meas((\cH_{x})_{x\in X},\pi)$.

A $*$-representation $\phi\colon L^{\infty}(X,\mu)\to B(\cH)$ is \emph{normal} if it is weak$^{*}$-WOT continuous.
Our goal in this section is to identify unitary representations of $\cR$ on with $*$-representations of $\C(\cR)$ which are normal on $L^{\infty}(X)$. The following lemma shows us that such $*$-representations of $\C(\cR)$ are automatically continuous on $[\cR]$, this will help us in later (see the proof of Proposition \ref{prop: rep correspond}) to reduce our understanding of representations of $\cR$ to $*$-representations of a subgroup of $[\cR]$ which generates the relation.

\begin{lem}\label{lem: implied continuity}
Let $\cH$ be a Hilbert space and $\pi\colon \C(\cR)\to B(\cH)$ a $*$-representation so that $\pi|_{L^{\infty}(X)}$ is normal. Then $\pi|_{[[\cR]]}$ is continuous if we give $[[\cR]]$ the metric 
\[d(\gamma,\sigma)=\mu((\dom(\gamma)\cup \dom(\sigma))\setminus\{x\in \dom(\gamma)\cap \dom(\sigma):\gamma(x)=\sigma(x)\}),\]
and $B(\cH)$ the strong operator topology.
\end{lem}

\begin{proof}
Suppose $\gamma_{n}\in [[\cR]]$ is a sequence and $\gamma\in [[\cR]]$ has $d(\gamma_{n},\gamma)\to_{n\to\infty}0$. Set $E_{n}=\{x\in \ran(\gamma)\cap \ran(\gamma_{n}):\gamma^{-1}(x)=\gamma_{n}^{-1}(x)\}$.  Then $\gamma 1_{E_{n}}=\gamma_{n} 1_{E_{n}}$, so viewing $\gamma,\gamma_{n}\in \C(\cR)$ we have:
\[\gamma-\gamma_{n}=\gamma 1_{\ran(\gamma)\setminus E_{n}}+\gamma_{n} 1_{\ran(\gamma_{n})\setminus E_{n}}.\]
Thus for any $\xi\in \cH$, we have that 
\[\|(\pi(\gamma)-\pi(\gamma_{n}))\xi\|\leq \|\pi(1_{\ran(\gamma)\setminus E_{n}})\xi\|+\|\pi(1_{\ran(\gamma_{n})\setminus E_{n}})\xi\|=\ip{\pi(1_{\ran(\gamma)\setminus E_{n}})\xi,\xi}^{1/2}+\ip{\pi(1_{\ran(\gamma_{n})\setminus E_{n}})\xi,\xi}^{1/2}.\]
Using the change of variables $x=\gamma(y)$ we see that 
\[\mu(\ran(\gamma)\setminus E_{n})=\mu(\dom(\gamma)\setminus \{x\in\dom(\gamma)\cap \dom(\gamma_{n}):\gamma(x)=\gamma_{n}(x)\})\to_{n\to\infty}0.\]
Similarly, $\mu(\ran(\gamma_{n})\setminus E_{n})\to_{n\to\infty}0$. Thus $1_{\ran(\gamma)\setminus E_{n}},1_{\ran(\gamma_{n})\setminus E_{n}}\to_{n\to\infty}0$ in the weak$^{*}$-topology and so the above shows that $\|(\pi(\gamma)-\pi(\gamma_{n}))\xi\|\to_{n\to\infty}0$.

\end{proof}

To show that representations of $\C(\cR)$ whose restriction to $L^{\infty}(X)$ are normal, it will be helpful to go through representations of a related algebra.

Let $(X,\mu)$ be a standard probability space and $G\actson (X,\mu)$ a probability measure-preserving action with $G$ countable. We let $L^{\infty}(X)\rtimes_{\text{alg}}G$ be the collection of all formal sums $f=\sum_{g\in G}f_{g}g$ where $\{g\in G:f_{g}\ne 0\}$ is finite. We equip $L^{\infty}(X)\rtimes_{\text{alg}}G$ with the natural structure as a vector space over $\C$.  For $f=\sum_{g\in G}f_{g}g,k=\sum_{g\in G}k_{g}g$ we define
\[fk=\sum_{g}\left(\sum_{h}f_{h}\alpha_{h}(k_{h^{-1}g})\right)g,\]
\[f^{*}=\sum_{g}\alpha_{g}(\overline{f}_{g^{-1}})g.\]
It is direct to check that $L^{\infty}(X)\rtimes_{\text{alg}}G$ is a $*$-algebra with the above operations. 

Suppose $(\cH_{x})_{x\in X}$ is a measurable field of Hilbert spaces over a standard probability space $(X,\mu)$ and that $G$ is a countable group with $G\acston (X,\mu)$ by probability measure-preserving actions. Suppose for almost every $x\in X$ and all $g\in G$ we have a $c(g,x)\in \cU(\cH_{x},\cH_{gx})$ with the property that $x\mapsto \ip{c(g,x)\xi_{x},\eta_{gx}}$ is measurable whenever $(\xi_{x})_{x\in X},(\eta_{x})_{x\in X}\in \Meas((\cH_{x})_{x\in X})$.
We call $c$ a \emph{cocycle} if for almost every $x\in X$ we have $c(g,hx)c(h,x)=c(gh,x)$ for all $g,h\in F$.  

\begin{prop}\label{prop:cocycle intro}
Let $(X,\mu)$ be a standard probability space and $G\actson (X,\mu)$ a probability measure-preserving action with $G$ countable. If $\cH$ is a separable Hilbert space and $\pi\colon L^{\infty}(X)\rtimes_{\text{alg}}G\to B(\cH)$ is a $*$-homomorphism with $\pi|_{L^{\infty}(X)}$ being normal, then we may find a measurable field $(\cH_{x})_{x\in X}$ of Hilbert spaces, a unitary $U\in \cU\left(\cH,\int_{X}^{\oplus}\cH_{x}\,d\mu(x)\right)$ and a cocycle $c(g,x)\in \cU(\cH_{x},\cH_{gx})$ so that 
\[U\pi(f)U^{*}=\int_{X}^{\oplus}f(x)\,d\mu(x), \text{ for all $f\in L^{\infty}(X)$,}\]
\[(U\pi(g)U^{*}\xi)_{x}=c(g,g^{-1}x)\xi_{g^{-1}x}, \text{ almost everywhere, for all $g\in G,\xi=(\xi_{x})_{x\in X}\in \int_{x}^{\oplus}\cH_{x}\,d\mu(x)$.}\]
\end{prop}

\begin{proof}

First, note that there is a unique (modulo null sets) measurable $E\subseteq X$ so that $\ker(\pi|_{L^{\infty}(X)})=\{f\in L^{\infty}(X):f1_{X\setminus E}=f\}$. Indeed, fixing a countable dense sequence $(\xi_{n})_{n}$ the fact that $\pi|_{L^{\infty}(X)}$ is normal implies that there are $k_{n}\in L^{1}(X)$ with $\ip{\pi(f)\xi_{n},\xi_{n}}=\int fk_{n}\,d\mu(x)$. If we set $E=\bigcup_{n}\{x:k_{n}(x)\ne 0\}$, we see that $E$ has the desired property. By \cite[Theorem VIII.4.8]{Conway} we have that $\pi|_{L^{\infty}(E)}$ is isometric. 

We now claim that $\pi(L^{\infty}(X))$ is a von Neumann algebra. Indeed, by Kaplansky's density theorem (see \cite[Theorem 44.1]{ConwayOT}),
\[\Ball(\overline{\pi(L^{\infty}(X))}^{WOT})=\Ball(\overline{\pi(L^{\infty}(E))}^{WOT})=\overline{\Ball(\pi(L^{\infty}(E)))}^{WOT}.\]
But since  $\pi|_{L^{\infty}(E)}$ is isometric, $\Ball(\pi(L^{\infty}(E))=\pi(\Ball(L^{\infty}(E)))$ and $\Ball(L^{\infty}(E))$ is weak$^{*}$ compact, the weak$^{*}$-WOT continuity of $\pi|_{L^{\infty}(X)}$ forces 
\[\Ball(\overline{\pi(L^{\infty}(X))}^{WOT})=\Ball(\pi(L^{\infty}(E))).\]
Thus $\overline{\pi(L^{\infty}(X))}^{WOT}=\pi(L^{\infty}(E))=\pi(L^{\infty}(X)).$

Thus we may regard $\pi|_{L^{\infty}(E)}$ as an isomorphism between the von Neumann algebras $L^{\infty}(E)$ and $\pi(L^{\infty}(X))$. 
It then follows from \cite[Propostion 11 in Part 2, Chapter 3, Section 6]{DixmierW*} that we may assume  $\cH=\int_{E}^{\oplus}\cH_{x}\,d\mu(x)$ for some measurable field $(\cH_{x})_{x\in E}$ of Hilbert spaces, and that $\pi(f)=\int_{X}^{\oplus}f(x)\,d\mu(x)$ for all $f\in L^{\infty}(E)$. We set $\cH_{x}=0$ for $x\in X\setminus E$ so that we still have $\cH=\int_{E}^{\oplus}\cH_{x}\,d\mu(x)$, and $\pi(f)=\int_{X}^{\oplus}f(x)\,d\mu(x)$ for all $f\in L^{\infty}(X)$.  It remains to construct the cocycle $c$.

\emph{Claim: the function $x\mapsto \dim(\cH_{x})$ is almost every $G$-invariant.}
To prove the claim it suffices to show that for all $n\in \N$, the set $E_{n}=\{x\in X:\dim(\cH_{x})\geq n\}$ is $G$-invariant modulo null sets. To see this, note by \cite[Proposition 1 of Part II, Chapter 1, Section 4]{DixmierW*} we may find measurable vector fields $(e^{(j)}_{x})_{x\in X}\in \Meas((\cH_{x})_{x\in X})$ for $j=1,\cdots,n$ so that $\pi(1_{E_{n}})e^{(j)}=e^{(j)}$ and $\ip{e^{(j)}_{x},e^{(k)}_{x}}=\delta_{j=k}$ for all $x\in E_{n}$. For $g\in G$, write $\pi(g)e^{(j)}=(\xi^{(j)}_{x})_{x\in X}\in \Meas((\cH_{x})_{x\in x})$. Since $\pi$ is a $*$-homomorphism, we have that $\pi(1_{gE_{n}})\xi^{(j)}=\xi^{(j)}$ for all $j=1,\cdots,n$ and for all $f\in L^{\infty}(X)$,
\begin{align*}
 \int f(x)\ip{\xi^{(j}_{x},\xi^{(k)}_{x}}\,d\mu(x)=\ip{\pi(f)\xi^{(j)},\xi^{(k)}}=\ip{\pi(\alpha_{g^{-1}}(f))e^{(j)},e^{(k)}}&=\delta_{j=k}\int_{E_{n}}f(gx)\,d\mu(x)\\
 &=\delta_{j=k}\int_{gE_{n}}f\,d\mu.   
\end{align*}
In particular, applying this for $f=1_{F}$ where $F$ ranges over all measurable subsets of $gE_{n}$ we see that $\ip{\xi^{(j)}_{x},\xi^{(k)}_{x}}=\delta_{j=k}$ for almost every $x\in gE_{n}$. Thus for almost every $x\in gE_{n}$ we have that $\dim(\cH_{x})\geq n$. Thus $\mu(gE_{n}\setminus E_{n})=0$ for all $g\in G$, proving the claim.

Having shown the claim, we now prove the theorem. We may partition $X$ into the countably many $G$-invariant sets $X_{n}=\{x\in X:\dim(\cH_{x})=n\}$ for $n\in \N\cup\{\infty\}$. If we define the appropriate cocycle on each $X_{n}$, then by $G$-invariance of each $X_{n}$ this produces the appropriate cocycle on $X$. So we may assume that $\dim(\cH_{x})$ is almost everywhere constant. In this case, by \cite[Proposition 3 of Part II, Chapter 1, Section 4]{DixmierW*} we may assume that there is a fixed Hilbert space $\widetilde{\cH}$ with $\cH_{x}=\widetilde{\cH}$, so $\cH=L^{2}(X,\widetilde{\cH})$. For $g\in G$, consider the operator $V_{\alpha_{g}}\in B(\cH)$ given by $(V_{\alpha_{g}}k)(x)=k(g^{-1}x)$. The since $\pi(g)\pi(f)\pi(g^{-1})=\pi(\alpha_{g}(f))$ for all $f\in L^{\infty}(X)$, we have that $V_{\alpha_{g}}^{*}\pi(g)$ commutes with $\pi(L^{\infty}(X))$. Hence, by \cite[Theorem 1 of Part II, Chapter 2, Section 5]{DixmierW*} we may find a measurable assignment $x\mapsto U_{g,x}\in \cU(\widetilde{\cH})$ so that $V_{\alpha_{g}}^{*}\pi(g)=\int_{X}^{\oplus}U_{g,x}\,d\mu(x):=U_{g}$. Then, for $k\in L^{2}(X,\widetilde{\cH})$
\[(\pi(g)k)(x)=(V_{\alpha_{g}}U_{g}k)(x)=(U_{g}k)(g^{-1}x)=U_{g,g^{-1}x}k(g^{-1}x).\]
Set $c(g,x)=U_{g,x}$. Then the fact that $\pi(g)\pi(h)=\pi(gh)$ and $(\pi(g)k)(x)=c(g,g^{-1}x)k(g^{-1}x)$ almost everywhere for all $g\in G,k\in L^{2}(X,\widetilde{\cH})$ forces $c(gh,x)=c(g,hx)c(h,x)$ almost everywhere for all $g,h \in G$ (e.g. by checking on $k_{j}(x)=\zeta^{(j)}$ where $(\zeta^{(j)})_{j\in J}$ is a dense sequence in $\widetilde{\cH}$).

\end{proof}

We now show that the construction described after Definition \ref{defn: groupoid ring} gives a complete description (modulo isomorphism) of all $*$-representations of $\C(\cR)$ whose restriction to $L^{\infty}(X)$ is normal.

The following is likely well known, but unable to find an appropriate reference, we include a proof for completeness.

\begin{prop}\label{prop: rep correspond}
Let $(\cR,X,\mu)$ be a discrete, probability measure-preserving equivalence relation. 
\begin{enumerate}[(i)]
\item \label{item:from unitary reps of R to *-reps} If $((\cH_{x})_{x\in X},\pi)$ is a unitary representation of $\cR$, then $\pi\colon \C(\cR)\to B(\cH)$ given by $(\pi(k)\xi)_{x}=\sum_{y}k(x,y)\pi(x,y)\xi_{y}$ if $\xi=(\xi_{x})_{x}\in \int_{X}^{\oplus}\cH_{x}\,d\mu(x)$ is a $*$-representation whose restriction to $L^{\infty}(X)$ is normal.
\item \label{item: from *-reps to unitary reps } Suppose $\rho\colon \C(\cR)\to B(\cK)$ is a $*$-representation on a Hilbert space $\cK$ so that $\rho|_{L^{\infty}(X)}$ is normal. Then there is a unitary representation $((\cH_{x})_{x\in X},\pi)$ of $\cR$ and a $U\in \cU(\cH,\cK)$ (where $\cH=\int_{X}^{\oplus}\cH_{x}\,d\mu(x)$) which satisfies 
\[\rho(f)=U\pi(f)U^{*} \text{ for all $f\in \C(\cR)$.}\]
\end{enumerate}
\end{prop}

\begin{proof}

(\ref{item:from unitary reps of R to *-reps}): We have already remarked that $\pi$ is a $*$-homomorphism. To see that $\pi|_{L^{\infty}(X)}$ is normal, it suffices to note that 
\[\ip{\pi(f)\xi,\eta}=\int_{X}f(x)\ip{\xi_{x},\eta_{x}}\,d\mu(x) \textnormal{ for all $f\in L^{\infty}(X)$},\]
and the function $x\mapsto \ip{\xi_{x},\eta_{x}}$ is in $L^{1}$ due to the Cauchy-Schwarz inequality. 

(\ref{item: from *-reps to unitary reps }):By \cite[Theorem 1]{FelMooreI} we may find a  countable group $G\subseteq [\cR]$ so that each element of $G$ is Borel, and which satisfies $\cR=\cR_{G,X}$. We then have a natural $*$-homomorphism $q\colon L^{\infty}(X)\rtimes_{\text{alg}}G\to \C(\cR)$ which is the identity when restricted to $L^{\infty}(X)$, $G$. By Proposition \ref{prop:cocycle intro}, it then follows that we may assume that we have a measurable field $(\cH_{x})_{x\in X}$ of Hilbert spaces,  that $\cK=\cH=\int_{X}^{\oplus}\cH_{x}$, and that there is a cocycle $c(g,x)\in \cU(\cH_{x},\cH_{gx})$ so that:
\begin{itemize}
\item $\pi(f)=\int_{X}^{\oplus} f(x)\,d\mu(x)$ for all $f\in L^{\infty}(X)$.
    \item $(\pi(q(g))\xi)_{x}=c(g,g^{-1}x)\xi_{g^{-1}x}$ for all $g\in G$.
\end{itemize}

\emph{Claim. For almost every $x\in X$ we have that $c(g,x)=c(h,x)$ for all $g,h\in G$ with $gx=hx$.}
By the cocycle condition, for almost every $x$ we have that if $gx=hx$, then
\[c(g,x)^{-1}c(h,x)=c(g^{-1},gx)c(h,x)=c(g^{-1},hx)c(h,x)=c(g^{-1}h,x).\]
So it suffices to show that for almost every $x\in X$ we have that $c(g,x)=\id$ for all $g\in \Stab(x)$. By the countability of $G$, it then suffices to show that $c(g,x)=x$ for almost every $x\in \Fix(g)$. Fix a measurable $E\subseteq \Fix(g)$. For $\xi=(\xi_{x})_{x},\zeta=(\zeta_{x})_{x}\in \int_{X}^{\oplus}\cH_{x}\,d\mu(x)$ we have
\begin{align*}
\int_{E}\ip{c(g,x)\xi_{x},\zeta_{x}}\,d\mu(x)=\int_{E}\ip{c(g,g^{-1}x)\xi_{g^{-1}x},\zeta_{x}}\,d\mu(x)=\ip{\pi(g)\xi,\pi(1_{E})\zeta}&=\ip{\pi(1_{E}g)\xi,\zeta}\\
&=\ip{\pi(1_{E})\xi,\zeta}\\
&=\int_{E}\ip{\xi_{x},\zeta_{x}}\,d\mu(x),
\end{align*}
where the first equality uses that $x\in \Fix(g)$ implies $g^{-1}x=x$ and the second-to-last equality uses that $1_{E}g=1_{E}$ in $\C(\cR)$ whenever $E\subseteq \Fix(g)$ is measurable. Since this is true for all $E\subseteq \Fix(g)$ measurable, we conclude that $\ip{c(g,x)\xi_{x},\zeta_{x}}=\ip{\xi_{x},\zeta_{x}}$ for almost every $x\in \Fix(g)$. Letting $\xi,\zeta$ run over a family of measurable vector fields which fiberwise generate $\cH_{x}$ for almost every $x$, we see that $c(g,x)=\id$ for almost every $x\in \Fix(g)$. This proves the claim.

Having proven the claim, it follows that for almost every $y\in X$, and for all $x\in[y]_{\cR}$ we have a well-defined unitary $\pi(x,y)\in \cU(\cH_{x},\cH_{y})$ given by $\pi(x,y)=c(g,y)$ where $gy=x$. It is then direct to show that $\pi$ is a unitary representation of $\cR$. We may thus define a $*$-representation $\widetilde{\pi}\colon \C(\cR)\to B(\cH)$ by $(\widetilde{\pi}(k)\xi)_{x}=\sum_{y}k(x,y)\pi(x,y)\xi_{y}$ and we have that $\widetilde{\pi}|_{L^{\infty}(X)}$ is weak$^{*}$-continuous. Moreover, $\widetilde{\pi}(k)=\pi(k)$ for $k\in L^{\infty}(X)\cup G$, by design. For arbitrary $\gamma\in [\cR]$, we may choose sets $(E_{g})_{g\in G}$ which partition $X$ modulo null sets and satisfy $E_{g}\subseteq\{x\in X:\gamma(x)=gx\}$. Write $G=\bigcup_{n}F_{n}$ where $F_{n}\subseteq F_{n+1}$ are nonempty, finite subsets of $G$, and set $\gamma_{n}=\gamma|_{\bigcup_{g\in F_{n}}E_{g}}$. Then $\gamma_{n}=\sum_{g\in F_{n}}g1_{E_{g}}$ in $\C(\cR)$ and by Lemma \ref{lem: implied continuity} we have that $\pi(\gamma_{n})\to\pi(\gamma)$, $\widetilde{\pi}(\gamma_{n})\to \widetilde{\pi}(\gamma)$ in the strong operator topology. Since $\pi(\gamma_{n})=\widetilde{\pi}(\gamma_{n})$ for every $n$, and $\gamma$ is arbitrary, we conclude that $\pi=\widetilde{\pi}$ by \cite[Lemma 3.3]{SauerBetti}.

\end{proof}

Motivated by the above, we define $\|\cdot\|$ on $\C(\cR)$ by saying that 
\[\|k\|=\sup\{\|\pi(k)\|:\pi\colon \C(\cR)\to B(\cH) \text{ is a $*$-representation so that $\pi|_{L^{\infty}(X)}$ is normal}\}.\]
The discussion following Definition \ref{defn: groupoid ring} shows that $\|k\|<+\infty$ for every $k\in \C(\cR)$. Moreover, letting $\lambda$ be as in Example \ref{ex: left/right reg rep} we have that $(\lambda(k)\xi)_{x}(y)=k(x,y)$ where $\xi=(\xi_{x})_{x}$ is the measurable vector field $\xi_{x}=\delta_{x}$. Thus $\|\cdot\|$ is indeed a norm on $\C(\cR)$.
We let $C^{*}(\cR)$ be the completion of $\C(\cR)$ under this norm. Note that $C^{*}(\cR)$ is a $C^{*}$-algebra.

\section{Weak containment for representations of $\cR$} \label{sec: weak containment}

\subsection{General facts about weak containment}
We recall here the notion of weak containment for representations of $C^{*}$-algebra.

\begin{thm}
Let $A$ be a unital $C^{*}$-algebra, $\cH_{j},j=1,2$ Hilbert spaces and $\pi_{j}\colon A\to B(\cH_{j}),j=1,2$ two $*$-representations. Then the following are equivalent 
\begin{enumerate}[(i)]
\item $\|\pi_{1}(a)\|\leq \|\pi_{2}(a)\|$ for all $a\in A$,
\item $\ker(\pi_{2})\subseteq \ker(\pi_{1})$,
\item for all $\xi\in \cH_{1}$, we have that 
\[\varphi_{\xi}\in \overline{\Span}^{wk^{*}}\{\varphi_{\zeta}:\zeta\in \cH_{2}\},\]
\item for all $\xi\in \cH_{1}$, we have that 
\[\varphi_{\xi}\in \overline{\co}^{wk^{*}}\{\varphi_{\zeta}:\zeta\in \cH_{2},\|\zeta\|=\|\xi\|\}.\]
\end{enumerate}
\end{thm}

\begin{proof}
The equivalence of the last three is contained in (\cite[Theorem 3.4.4. and  Proposition 3.4.9]{DixmierC*}). The fact that the first implies the second is an exercise. To see that the second implies the first, note that our hypothesis implies that $\pi_{1}=\Psi\circ \pi_{2}$ for a $*$-homomorphism $\Psi\colon \pi_{2}(A)\to \pi_{1}(A)$. Since $*$-homomorphisms of $C^{*}$-algebras are contractive (see \cite[Proposition VIII.1.11 (d)]{Conway}) it follows that 
\[\|\pi_{1}(a)\|=\|\Psi(\pi_{2}(a))\|\leq \|\pi_{2}(a)\|, \text{ for all $a\in A$}.\]
\end{proof}

In the setup of the above Theorem, we say that $\pi_{1}$ is weakly contained in $\pi_{2}$ and write $\pi_{1}\preceq \pi_{2}$ if any (equivalently all) of the above conditions hold. We say that $\pi_{1},\pi_{2}$ are weakly equivalent if $\pi_{1}\preceq \pi_{2}$, $\pi_{2}\preceq \pi_{1}$.

The following allows one to reduce checking weak containment to checking it on a set of vectors whose $A$-translates densely span $\cH_{1}$, which is frequently helpful. This proposition is surely well known (see e.g. the proof of \cite[Lemma F.1.3]{BHV} for the special case of groups), but being unable to find a suitable reference we include a proof for completeness. 

\begin{prop}\label{prop: linear reduction weak cont}
Let $A$ be a unital $C^{*}$-algebra, $\cH_{j},j=1,2$ Hilbert spaces and $\pi_{j}\colon A\to B(\cH_{j}),j=1,2$ two unital $*$-representations. Set $V_{\pi_{j}}=\overline{\Span}^{wk^{*}}\{\varphi_{\zeta}:\zeta\in \cH_{j}\}\subseteq A^{*},j=1,2$. Then
\[\{\xi\in \cH_{1}:\varphi_{\xi}\in V_{\pi_{2}}\}\]
is a  closed, linear, $A$-invariant subspace of $\cH_{1}$.
\end{prop}

\begin{proof}
Let $W\subseteq \cH$ be the set of $\xi\in \cH_{1}$ so that $\varphi_{\xi}\in V_{\pi_{2}}$. It is direct to show that $W$ is closed, $\pi_{1}(A)$ invariant and closed under scaling. It remains to show that $W$ is closed under sums. Let $\xi_{1},\xi_{2}\in W$. Write $\xi_{2}=\xi_{2}'+\xi_{2}''$ where $\xi_{2}'\in \overline{\pi_{1}(A)\xi_{1}}^{\|\cdot\|}$, and $\xi_{2}''\perp \overline{\pi_{1}(A)\xi_{1}}^{\|\cdot\|}.$ Then $\xi_{1}+\xi_{2}'\in \overline{\pi_{1}(A)\xi_{1}}$, and so $\xi_{1}+\xi_{2}'\in W$, since $\xi_{1}\in W$.  Moreover,
\[\varphi_{\xi_{2}''}=\varphi_{\xi_{2}}-\varphi_{\xi_{2}'}\in V_{\pi_{2}},\]
so $\xi_{2}''\in W$. Since
\[\varphi_{\xi_{1}+\xi_{2}}=\varphi_{\xi_{1}+\xi_{2}'}+\varphi_{\xi_{2}''},\]
it follows that $\xi_{1}+\xi_{2}\in W$. 
\end{proof}

For unitary representations of equivalence relations, there is a separate notion of weak containment.

\begin{defn}
Let $\cR$ be a discrete, probability measure-preserving, equivalence relation on a standard probability space $(X,\mu)$. Let $((\cH_{x})_{x\in X},\pi),((\cK_{x})_{x\in X},\rho)$ be two unitary representations of $\cR$. We say that $\pi\preceq \rho$ if for all measurable $Q\subseteq \cR$ with $\mu_{\cR}(Q)<+\infty$, and for all $\xi\in \Meas((\cH_{x})_{x\in X})$, and $\varepsilon>0$, there are $\zeta_{1},\cdots,\zeta_{k}\in \Meas((\cK_{x})_{x\in X})$ with
\[\mu_{\cR}\left(\left\{(x,y)\in Q:\left|\ip{\pi(x,y)\xi_{y},\xi_{x}}-\sum_{i=1}^{k}\ip{\rho(x,y)\zeta_{i,y},\zeta_{i,x}}\right|>\varepsilon\right\}\right)<\varepsilon.\]
\end{defn}

In Proposition \ref{prop: rep correspond}, we showed that unitary representations of $\cR$ correspond to representations of $\C(\cR)$ which are normal when restricted to $L^{\infty}(X)$. Hence unitary representations of $\cR$ correspond to certain representations of $C^{*}(\cR)$. We start this section by showing that the notions of weak containment as representations of $\cR$ and as representations of $C^{*}(\cR)$ coincide. 
For this, it will be helpful to check weak$^{*}$-convergence on $\C(\cR)$ rather than on all of $C^{*}(\cR)$. We will use the following proposition to do this. 

\begin{prop}\label{prop:dense set reduction}
Let $A$ be a unital $C^{*}$-algebra, $\cH_{j},j=1,2$ Hilbert spaces and let $\pi_{j}\colon A\to B(\cH_{j})$ be two $*$-homomorphisms. Suppose that are sets $B\subseteq A$ and $D\subseteq \cH_{1}$ so that $\overline{\Span}^{\|\cdot\|}(B)=A$ and $\overline{\Span}^{\|\cdot\|}(\pi_{1}(A)D)=\cH_{1}$. Assume that for all $\xi\in D$, there is a constant $C_{\xi}\geq 0$ so that for all $b_{1},\cdots,b_{k}\in B$ we have 
\[(\varphi_{\xi}(b_{j}))_{j=1}^{k}\in \overline{\co}\{(\varphi_{\zeta}(b_{j}))_{j=1}^{k}:\zeta\in \cH_{2},\|\zeta\|\leq C_{\xi}\}.\]
Then $\pi_{1}\preceq \pi_{2}$. 

\end{prop}

\begin{proof}
By Proposition \ref{prop: linear reduction weak cont}, it suffices to show that for all $\xi\in D$ we have that 
\[\varphi_{\xi}\in \overline{\co}^{wk^{*}}\{\varphi_{\zeta}:\zeta\in \cH_{2},\|\zeta\|\leq C_{\xi}\}.\]
Set $K=\overline{\co}^{wk^{*}}\{\varphi_{\zeta}:\zeta\in \cH_{2},\|\zeta\|\leq C_{\xi}\}$. For $k\in \N$, let
\[\Upsilon_{k}=\{(x_{j})_{j=1}^{k}\in A^{k}:(\varphi_{\xi}(x_{j}))_{j=1}^{k}\in \overline{\{(\varphi(x_{j}))_{j=1}^{k}:\varphi\in K\}}\}.\]
It suffices to show that $\Upsilon_{k}=A^{k}$ for all $k\in \N$. 
Since $\|\varphi\|\leq C_{\xi}^{2}$ for all $\varphi\in K$, it follows that $\Upsilon_{k}$ is a closed subset of $A^{k}$ for each $k\in \N$.
Our assumptions imply that $\Upsilon_{k} \supseteq (\Span(B))^{k}$, and thus $\Upsilon_{k}\supseteq \overline{(\Span(B))^{k}}=A^{k}$. 
\end{proof}

If $((\cH_{x})_{x\in X},\pi)$ is a unitary representation of $\cR$, and $\xi\in \int_{X}^{\oplus}\cH_{x}\,d\mu(x)$, $\gamma\in [\cR]$, we define $\omega_{\xi,\gamma}^{\pi}\colon X\to \C$ by $\omega_{\xi,\gamma}^{\pi}(x)=\ip{\pi(x,\gamma^{-1}(x))\xi_{\gamma^{-1}(x)},\xi_{x}}$. By Cauchy-Schwartz, we see that $\omega_{\xi,\gamma}^{\pi}\in L^{1}(X)$.
We now prove that weak containment of unitary representations of $\cR$ coincides with weak containment of the induced representations of $C^{*}(\cR)$. We remark that our method follows the same strategy as \cite[Proposition 4.6]{floresharbourgroupoid}.

\begin{lem}
Let $((\cH_{x})_{x\in X},\pi),((\cK_{x})_{x\in X},\rho)$ be two unitary representations of $\cR$, and regard $\cH=\int_{X}^{\oplus}\cH_{x}\,d\mu(x),\cK=\int_{X}^{\oplus}\cK_{x}\,d\mu(x)$ as representations of $C^{*}(\cR)$.
Then $\cH \preceq \cK$ as representations of $\cR$ if and only if $\cH\preceq \cK$ as representations of $C^{*}(\cR)$.  
\end{lem}

\begin{proof}
First suppose that  $\cH\preceq \cK$ as representations of $C^{*}(\cR)$. Let $\xi=(\xi_{x})_{x}\in \int_{X}^{\oplus}\cH_{x}\,d\mu(x)$, and $\gamma_{1},\cdots,\gamma_{n}\in [\cR]$. The fact that    $\cH\preceq \cK$ implies that if we view $(\omega_{\xi,\gamma_{j}}^{\pi})_{j=1}^{n}\in L^{1}(X)^{\oplus n}$ and give $L^{1}(X)^{\oplus n}$ the $\ell^{1}$-direct sum of the $\|\cdot\|_{L^{1}(X)}$ norm, then $(\omega_{\xi,\gamma_{j}}^{\pi})_{j=1}^{n}$ is in the weak closure of 
\[ \co\{(\omega_{\zeta,\gamma_{j}}^{\rho})_{j=1}^{n}:\zeta\in \cK,\|\zeta\|\leq \|\xi\|\}.\]
Thus by separating Hahn-Banach (see \cite[Theorem V.1.4]{Conway}), we have that 
\[(\omega_{\xi,\gamma_{j}}^{\pi})_{j=1}^{n}\in \overline{\co}^{\|\cdot\|_{1}}\{(\omega_{\zeta,\gamma_{j}}^{\rho})_{j=1}^{n}:\zeta\in \cK,\|\zeta\|\leq \|\xi\|\}.\]
So given $\varepsilon>0$, we may find $\zeta_{1},\cdots,\zeta_{k}\in \cK$ with $\|\zeta_{j}\|\leq \|\xi\|$, and $\lambda_{1},\cdots,\lambda_{k}\in [0,1]$ with $\sum_{i=1}^{k}\lambda_{i}=1$ and which satisfy
\[\sum_{j=1}^{n}\left\|\omega_{\xi,\gamma_{j}}^{\pi}-\sum_{i=1}^{k}\lambda_{i}\omega_{\zeta_{i},\gamma_{j}}^{\rho}\right\|_{1}<\varepsilon^{2}\]
This implies that 
\begin{align*}
&\int_{\bigcup_{j=1}^{n}\gr(\gamma_{j}^{-1})}\left|\ip{\pi(x,y)\xi_{y},\xi_{x}}-\sum_{i=1}^{k}\lambda_{i}\ip{\rho(x,y)\zeta_{i,y},\zeta_{i,x}}\right|\,d\mu(x)\\
&\leq \sum_{j=1}^{n}\int_{\gr(\gamma_{j}^{-1})}\left|\ip{\pi(x,y)\xi_{y},\xi_{x}}-\sum_{i=1}^{k}\lambda_{i}\ip{\rho(x,y)\zeta_{i,y},\zeta_{i,x}}\right|\,d\mu(x)<\varepsilon^{2}.
\end{align*}
So by Chebyshev's inequality,
\[\mu_{\cR}\left(\left\{(x,y)\in \bigcup_{j=1}^{n}\gr(\gamma_{j}^{-1}):\left|\ip{\pi(x,y)\xi_{y},\xi_{x}}-\sum_{i=1}^{k}\lambda_{i}\ip{\rho(x,y)\zeta_{i,y},\zeta_{i,x}}\right|>\varepsilon\right\}\right)<\varepsilon.\]

Now suppose that $\xi\in \Meas((\cH_{x})_{x\in X})$, a finite measure $Q\subseteq \cR$, and an $\varepsilon>0$ are given. Choose $\gamma_{1},\cdots,\gamma_{n}\in [\cR]$ so that 
\[\mu_{\cR}(Q\setminus \bigcup_{j=1}^{n}\gr(\gamma_{j})^{-1})<\varepsilon.\]
Then choose $\widetilde{\xi}=(\widetilde{\xi}_{x})_{x}\in \cH$ so that $E:=\{x:\widetilde{\xi}_{x}\ne \xi_{x}\}$ has $\mu(E)<\frac{\varepsilon}{n}$. By the above, we may find $\zeta_{1},\cdots,\zeta_{k}\in \cK$ with
\[\mu_{\cR}\left(\left\{(x,y)\in \bigcup_{j=1}^{n}\gr(\gamma_{j}^{-1}):\left|\ip{\pi(x,y)\widetilde{\xi}_{y},\widetilde{\xi}_{x}}-\sum_{i=1}^{k}\lambda_{i}\ip{\rho(x,y)\zeta_{i,y},\zeta_{i,x}}\right|>\varepsilon\right\}\right)<\varepsilon.\]
Set 
\[Q_{0}=Q\cap \bigcup_{j=1}^{n}\{(x,y)\in \gr(\gamma_{j}^{-1}(x)):y\in X\setminus E\}.\]
Then
\[\mu_{\cR}(Q\setminus Q_{0})<2\varepsilon,\]
and by the above
\[\mu_{\cR}\left(\left\{(x,y)\in Q_{0}:\left|\ip{\pi(x,y)\xi_{y},\xi_{x}}-\sum_{i=1}^{k}\lambda_{i}\ip{\rho(x,y)\zeta_{i,y},\zeta_{i,x}}\right|>\varepsilon\right\}\right)<\varepsilon.\]
Thus 
\[\mu_{\cR}\left(\left\{(x,y)\in Q:\left|\ip{\pi(x,y)\xi_{y},\xi_{x}}-\sum_{i=1}^{k}\lambda_{i}\ip{\rho(x,y)\zeta_{i,y},\zeta_{i,x}}\right|>\varepsilon\right\}\right)<3\varepsilon.\]
So $\cH\preceq \cK$ as representations of $\cR$.

Conversely, suppose that $\cH\preceq \cK$ as representations of $\cR$. Let $\xi=(\xi_{x})_{x}\in \cH$, and $f_{1},\cdots,f_{\ell}\in L^{\infty}(X),\gamma_{1},\cdots,\gamma_{n}\in [\cR]$, and $\varepsilon>0$ be given. We assume, without loss of generality that $\id=\gamma_{1}.$ Set $C=\max_{j}\|f_{j}\|_{\infty}$, and let $Q=\bigcup_{j=1}^{n}\gr(\gamma_{j}^{-1})$. Choose $\delta>0$ so that if $A\subseteq X$ is measurable with $\mu(A)<\delta$, then
\[\int_{A}\|\xi_{x}\|^{2}\,d\mu(x)<\frac{\varepsilon}{C+1}.\]

Since $\cH\preceq \cK$ as representations of $\cR$, we may find $\zeta_{1},\zeta_{2},\cdots,\zeta_{k}\in \Meas((\cK_{x})_{x})$ so that if 
\[F=\left\{(x,y)\in Q:\left|\ip{\pi(x,y)\xi_{y},\xi_{x}}-\sum_{i=1}^{k}\ip{\rho(x,y)\zeta_{i,y},\zeta_{i,x}}\right|>\frac{\varepsilon}{C+1}\right\},\]
then $\mu_{\cR}(F)<\min(\delta/n,\varepsilon)$. Let $T=\{x\in X:(x,x)\in F\}$, and  $S=X\setminus T$. For $i=1,\cdots,k$, set $\widetilde{\zeta}_{i,x}=1_{S}(x)\zeta_{i,x}$. Since $\gamma_{1}=\id$, for $x\in S$ we have
\[\left\|\|\xi_{x}\|^{2}-\sum_{i=1}^{k}\|\zeta_{i,x}\|^{2}\right|<\frac{\varepsilon}{C+1}.\]
In particular, we obtain $\sum_{i=1}^{k}\|\widetilde{\zeta}_{i}\|^{2}\leq \|\xi\|^{2}+1.$ 
Set $C_{\xi}=\sqrt{\|\xi\|^{2}+1}$.
Letting $\pi_{2}\colon \cR\to X$ be the projection onto the second coordinate, we have that $\pi_{2}(F)$ is analytic, and hence $\mu$-measurable. Set $B=\bigcup_{j=1}^{n}\gamma_{j}(\pi_{2}(F))$. Note that $\mu(\pi_{2}(F))\leq \mu_{\cR}(F)<\delta/n.$ Thus our choice of $\delta$ implies that for all $1\leq s\leq \ell,1\leq j\leq n$ we have:
\[\int_{B}|f_{s}(x)||\ip{\xi_{\gamma_{j}^{-1}(x)},\xi_{x}}|\,d\mu(x)\leq C\left(\int_{B}\|\xi_{x}\|^{2}\,d\mu(x)\right)^{1/2}\left(\int_{\gamma_{j}^{-1}(B)}\|\xi_{x}\|^{2}\,d\mu(x)\right)^{1/2}<\varepsilon.\]
Similarly,
\begin{align*}
&\int_{B}|f_{s}(x)|\left|\sum_{i=1}^{k}\ip{\widetilde{\zeta}_{\gamma_{j}^{-1}(x)},\widetilde{\zeta}_{x}}\right|\,d\mu(x)\\
&\leq C\left(\int_{B\cap S}\|\xi_{x}\|^{2}+\frac{\varepsilon}{C+1}\,d\mu(x)\right)^{1/2}\left(\int_{\gamma_{j}^{-1}(B)\cap S}\|\xi_{x}\|^{2}+\frac{\varepsilon}{C+1}\,d\mu(x)\right)^{1/2}<2\varepsilon.
\end{align*}
Thus 
\begin{align*}
&\left|\ip{\pi(f_{s}\gamma_{j})\xi,\xi}-\sum_{i=1}^{k}\ip{\rho(f_{s}\gamma_{j})\widetilde{\zeta}_{i},\widetilde{\zeta_{i}}}\right|\\
&\leq 3\varepsilon+C\int_{X\setminus B}\left|\ip{\pi(x,\gamma_{j}^{-1}(x))\xi_{\gamma_{j}^{-1}(x)},\xi_{x}}-\sum_{i=1}^{k}\ip{\rho(x,\gamma_{j}^{-1}(x))\widetilde{\zeta}_{i,\gamma_{j}^{-1}(x)},\widetilde{\zeta}_{i,x}}\right|\,d\mu(x).    
\end{align*}
For $x\in X\setminus B$, we have that $\gamma_{j}^{-1}(x)\notin \pi_{2}(F)$ and thus $\gamma_{j}^{-1}(x)\in S$ and $(x,\gamma_{j}^{-1}(x))\notin F$. Thus our choice of $F$ implies that 
\[\left|\ip{\pi(f_{s}\gamma_{j})\xi,\xi}-\sum_{i=1}^{k}\ip{\rho(f_{s}\gamma_{j})\widetilde{\zeta}_{i},\widetilde{\zeta_{i}}}\right|<4\varepsilon.\]
Since 
$\sum_{i=1}^{k}\|\widetilde{\zeta}_{i}\|^{2}\leq C_{\xi}^{2},$
we may write 
\[\sum_{i=1}^{k}\ip{\rho(\cdot)\widetilde{\zeta}_{i},\widetilde{\zeta}_{i}}\]
as a convex combination of vector states coming from vectors which have norm at most $C_{\xi}$. We are thus in a position to apply Proposition \ref{prop:dense set reduction}.

\end{proof}

\subsection{Coamenability and weak containment}
For a discrete, probability measure-preserving, relation $\cR$ on a standard probability space $(X,\mu)$, we regard $L^{2}(X)$ as a representation $\alpha$ of $\cR$ by regarding $L^{2}(X)=\int_{X}^{\oplus}\C\,d\mu(x)$ and setting $\alpha(x,y)=\id$. Then $\alpha(\gamma)=\alpha_{\gamma}$ for $\gamma\in [\cR]$ and $\alpha(f)\xi=f\xi$ for $f\in L^{\infty}(X),\xi\in L^{2}(X)$.

For $\cS\leq \cR$, recall the representation $\lambda_{\cR/\cS}$ discussed in Example \ref{ex: quasi-regular rep}. As discussed there, this is the analogue of the quasi-regular representation of a group on the coset space of a subgroup. Additionally, we recall that $\alpha$ is supposed to be the analogue of the trivial representation of a group. We thus should expect that, analogous to the group case, that coamenability should be the same as weak containment of $\alpha$ inside $\lambda_{\cR/\cS}$. We proceed to prove that this is the case. 

\begin{prop}\label{prop: coamenable weak containment characterize}
Let $\cS\leq \cR$ be discrete, probability measure-preserving, relations on a standard probability space $(X,\mu)$. Then $\cS\leq \cR$ is a coamenable inclusion if and only if $L^{2}(X)\preceq L^{2}(\cR/\cS)$ as representations of $\cR$.
\end{prop}

\begin{proof}
First, suppose that $\cS\leq \cR$ is a coamenable inclusion. Then, by \cite[Theorem 3.5]{HayesCoAmen} we may find $\xi^{(n)}\in L^{2}(\cR/\cS)$ with $\|\xi^{(n)}_{x}\|_{2}=1$ for almost every $x\in X$ and with 
\[\lim_{n\to\infty}\|\xi^{(n)}_{x}-\xi^{(n)}_{y}\|_{2}=0 \text{ for almost every $(x,y)\in\cR$.}\]
Then for every $k\in \C(\cR)$ we have 
\[\ip{\lambda_{\cR/\cS}(k)\xi^{(n)},\xi^{(n)}}=\int\sum_{y\in [x]_{\cR}}k(x,y)\ip{\xi^{(n)}_{y},\xi^{(n)}_{x}}\,d\mu(x).\]
Since $\|\xi^{(n)}_{x}\|_{2}=1$ almost everywhere, the dominated convergence theorems allows us to conclude that 
\[\lim_{n\to\infty}\ip{\lambda_{\cR/\cS}(k)\xi^{(n)},\xi^{(n)}}=\int_{X}\sum_{y}k(x,y)\,d\mu(x)=\ip{\alpha(k)1,1}.\]
Since $\overline{\{\alpha(k)1:k\in \C(\cR)\}}=L^{2}(X)$, and $\|\xi^{(n)}\|_{2}=1$, we apply Proposition \ref{prop:dense set reduction}  to conclude that $L^{2}(X)\preceq L^{2}(\cR/\cS)$.

For the reverse implication, let $\nu\in \Prob([\cR])$ be countably supported so that if $G$ is the subgroup of $[\cR]$ generated by $\supp(\nu)$, then $Gx=[x]_{\cR}$ for almost every $x\in X$. Let $E\subseteq X$ be an $[\cR]$-invariant measurable set with $\mu(E)>0$. Since $L^{2}(X)\preceq L^{2}(\cR/\cS)$, we have that 
\[\|\lambda_{\cR|_{E}/\cS|_{E}}(\nu)\|=\|\lambda_{\cR/\cS}(1_{E})\lambda_{\cR/\cS}(\nu)\lambda_{\cR/\cS}(1_{E})\|\geq \frac{1}{\mu(E)^{1/2}}\|\alpha_{\cR/\cS}(1_{E})\alpha_{\cR/\cS}(\nu)\alpha_{\cR/\cS}(1_{E})1_{E}\|=1,\]
the last line following from $\cR$-invariance of $E$. Hence, by \cite[Theorem 3.5]{HayesCoAmen} it follows that $\cS\leq \cR$ is coamenable.

\end{proof}

\subsection{Weak containment and almost invariant vectors}

As discussed in Example \ref{ex: trival rep}, for a discrete, probability measure-preserving equivalence relation $\cR$ on a standard probability space $(X,\mu)$, the representation $L^{2}(X)$ is the analogue of the trivial representation of a group. For unitary representations of groups, weakly containing the trivial representation is the same as having almost invariant vectors. In this subsection, we prove the appropriate analogue for representations of equivalence relations.

For this, it will be helpful to prove the following lemma, which is similar to what we did in \cite[Theorem 2.8]{HayesCoAmen}.

\begin{lem}\label{lem: inv set reduction}
Let $\cR$ be a discrete, probability measure-preserving, equivalence relation on a standard probability space $(X,\mu)$. Suppose that $\varphi\in C^{*}(\cR)^{*}$ is a state, and that $\varphi(1_{E})=\mu(1_{E})$ for every $\cR$-invariant measurable set $E$, and that $\varphi(\gamma)=1$ for every $\gamma\in [\cR]$. Then $\varphi(a)=\ip{\alpha(a)1,1}_{L^{2}(X)}$ for all $a\in C^{*}(\cR)$. 
    
\end{lem}

\begin{proof}
By \cite[Theorem VIII.5.14]{Conway}, we may find a Hilbert space $\cH_{\varphi}$, a unit vector $\xi_{\varphi}\in \cH_{\varphi}$ and a $*$-homomorphism $\pi_{\varphi}\colon C^{*}(\cR)\to B(H_{\varphi})$ with $\varphi(x)=\ip{\pi_{\varphi}(x)\xi_{\varphi},\xi_{\varphi}}$. Expanding $\|\pi_{\varphi}(\gamma)\xi_{\varphi}-\xi_{\varphi}\|^{2}$ and using that $\varphi(\gamma)=1$, we see that $\pi_{\varphi}(\gamma)\xi_{\varphi}=\xi_{\varphi}$. In particular, $\varphi(\gamma x)=\varphi(x)=\varphi(x\gamma)$ for every $x\in C^{*}(\cR)$ and $\gamma\in [\cR]$. 
 
Thus for all $f\in L^{\infty}(X)$ we have 
\[\varphi(\alpha_{\gamma}(f))=\varphi(\gamma f\gamma^{-1})=\varphi(f).\]
Hence, by the proof of \cite[Lemma 4.2]{HVTypeIIICartan} (see also \cite[Theorem 2.8]{HayesCoAmen}) our hypothesis implies that $\varphi(f)=\int_{X}f\,d\mu$ for all $f\in L^{\infty}(X)$. Thus for all $f\in L^{\infty}(X),\gamma\in [\cR]$ we have that 
\[\varphi(f\gamma)=\varphi(f)=\int f\,d\mu=\ip{\alpha(f\gamma)1,1}_{L^{2}(X)}.\]
Since $\overline{\Span\{f\gamma:f\in L^{\infty}(X),\gamma\in [\cR]\}}=C^{*}(\cR)$, the result follows by continuity of $\varphi$ and $\ip{\alpha(\cdot) 1,1}_{L^{2}(X)}$. 
\end{proof}

If $((\cH_{x})_{x\in X},\pi)$ is a unitary representation of a discrete, probability measures-preserving relation $\cR$ on $(X,\mu)$, then we can construct a new representation  $((\cH_{x}^{\oplus \N})_{x\in X},\pi^{\oplus \N})$. Here we equip $((\cH_{x}^{\oplus \N})_{x\in X}$ with the measurable structure so that $(\xi_{x})_{x}\in \prod_{x\in X}\cH_{x}^{\oplus \N}$ is measurable if and only if for every $j\in \N$ we have $(\xi_{x}(j))_{x\in X}\in \Meas((\cH_{x})_{x\in X})$. Moreover we set
$\pi^{\oplus \N}(x,y)=(\pi(x,y))^{\oplus \N}$. Note that we have a natural identification 
\[\left(\int_{X}^{\oplus} \cH_{x}\,d\mu(x)\right)^{\oplus \N}\cong \int_{X}^{\oplus}\cH_{x}^{\oplus \N}\,d\mu(x)\]
via
\[(\xi_{x})_{x\in X}\mapsto (j\mapsto (\xi_{x}(j)_{x\in X}).\]

\begin{prop}
Let $\cR$ be a discrete, probability measure-preserving, equivalence relation on a standard probability space $(X,\mu)$. Let $((\cH_{x})_{x\in X},\pi)$ be a unitary representation of $\cR$, and set $\cH=\int_{X}^{\oplus}\cH_{x}\,d\mu(x)$. The following are equivalent.
\begin{enumerate}[(i)]
\item \label{item: weak containment} $L^{2}(X)$ is weakly contained in $\cH$,
\item \label{item: a.i. measurable sets} for every $\cR$-invariant measurable set $E$, we have that $L^{2}(E)$ has $[\cR]$-almost invariant vectors. 
\item \label{item: a.i. in inf direct sum} there is a sequence $\xi^{(n)}\in \cH^{\oplus \N}$ with $\|\xi^{(n)}_{x}\|_{2}=1$ for almost every $x\in X$ and so that 
\[\|\pi^{\oplus \N}(x,y)\xi^{(n)}_{y}-\xi^{(n)}_{x}\|_{2}\to_{n\to\infty}0 \text{ for almost every $(x,y)\in\cR$.}\]
\end{enumerate}

\end{prop}

\begin{proof}
Throughout the proof, we let $\psi\in C^{*}(\cR)^{*}$ be the state associated with the vector $1\in L^{2}(X)$, i.e.
\[\psi(a)=\ip{\alpha(a)1,1}_{L^{2}(X)}.\]

(\ref{item: weak containment}) implies (\ref{item: a.i. in inf direct sum}): 
For $\xi\in \cH^{\oplus \N}$, define $\omega_{\xi}\in L^{1}(X)$ by $\omega_{\xi}(x)=\|\xi_{x}\|^{2}$. Given $\gamma_{1},\cdots,\gamma_{n}\in [\cR]$, set 
\[K=\{((\varphi_{\xi}(\gamma_{j})-1)_{j=1}^{n},\omega_{\xi}-1):\xi\in \cH^{\oplus N},\|\xi\|\leq 1\}\subseteq \C^{n}\oplus L^{1}(X).\]

We start by proving the following

\emph{Claim 1: $K$ is convex.}
To see that this is true, let $\xi_{1},\xi_{2}\in\cH^{\oplus \N}$ with $\|\xi_{j}\|\leq 1$. Then for $\lambda\in [0,1]$ we have $\sqrt{\lambda}\xi_{1}\oplus \sqrt{1-\lambda}\xi_{2}\in \cH^{\oplus (\N\sqcup \N)}$ has norm $1$, and 
\[\varphi_{\sqrt{\lambda}\xi_{1}\oplus \sqrt{1-\lambda}\xi_{2}}=\lambda\varphi_{\xi_{1}}+(1-\lambda)\varphi_{\xi_{2}},\]
\[\omega_{\sqrt{\lambda}\xi_{1}\oplus \sqrt{1-\lambda}\xi_{2}}=\lambda\omega_{\xi_{1}}+(1-\lambda)\omega_{\xi_{2}},\]
the claim then follows by bijecting $\N\sqcup\N$ with $\N$.

We now prove  
\emph{Claim 2: $0\in \overline{K}^{\|\cdot\|}$.}
To see this, let $\varepsilon>0$ and $f_{1},\cdots,f_{k}\in L^{\infty}(X)$ be given.
Then for $\xi\in \cH^{\oplus \N}$ we have 
\[\int f\omega_{\xi}\,d\mu-\int f\,d\mu=\ip{\pi^{\oplus \N}(f)\xi,\xi}-\ip{\alpha(f)\xi,\xi},\]
and 
\[\varphi_{\xi}(\gamma_{j})-1=\ip{\pi^{\oplus \N}(\gamma_{j})\xi,\xi}-\ip{\alpha(\gamma_{j})\xi,\xi}.\]
Hence by Claim 1 and the fact that $L^{2}(X)$ is weakly contained in $\cH$, we see that 
\[0\in \overline{K}^{wk}=\overline{K}^{\|\cdot\|},\]
the last equality following by convexity of $K$ and separating Hahn-Banach (see \cite[Theorem IV.3.13]{Conway}).

Having shown Claims 1 and Claim 2, we now prove (\ref{item: a.i. in inf direct sum}). Let $\Gamma\leq [\cR]$ be countable and with $\cR=\cR_{\Gamma,x}$. By Claim 2 and a diagonal argument, we may find a sequence $\xi^{(n)}_{x}$ so that $\ip{\pi^{\oplus \N}(\gamma)\xi^{(n)},\xi^{(n)}}\to_{n\to\infty}1$ for all $\gamma\in \Gamma$, and $\|\omega_{\xi^{(n)}}-1\|_{1}\to_{n\to\infty}0$. Since  $\|\omega_{\xi^{(n)}}-1\|_{1}\to_{n\to\infty}0$, we may perturb $\xi^{(n)}$ and assume that $\|\xi^{(n)}_{x}\|_{2}=1$ for almost every $x\in X$. The fact that $\ip{\pi^{\oplus \N}(\gamma)\xi^{(n)},\xi^{(n)}}\to_{n\to\infty}1$ implies that $\|\pi^{\oplus \N}(\gamma)\xi^{(n)}-\xi^{(n)}\|_{2}\to_{n\to\infty}0$. Applying countability of $\Gamma$ and passing to a subsequence we may assume that for almost every $x\in X$ we have
\[\|\pi^{\oplus \N}(x,\gamma^{-1}(x))\xi^{(n)}_{\gamma^{-1}(x)}-\xi^{(n)}_{x}\|_{2}\to_{n\to\infty}0, \textnormal{ for all $\gamma\in \Gamma$.}\]
Since $\Gamma$ is a group and $\Gamma x=[x]_{\cR}$ for almost every $x\in X$, we conclude (\ref{item: a.i. in inf direct sum}).

(\ref{item: a.i. in inf direct sum}) implies (\ref{item: a.i. measurable sets}):  Fix $F\subseteq [\cR]$, and $\varepsilon>0$ be given. Let $\xi^{(n)}$ be as (\ref{item: a.i. in inf direct sum}). Write $\xi^{(n)}=(\xi^{(n)}_{j})_{j=1}^{\infty}$ with $\xi^{(n)}_{j}=(\xi^{(n)}_{j,x})_{x\in X}$. Note that
\[\lim_{n\to\infty}\sum_{j=1}^{\infty}\sum_{\gamma\in F}\|\pi(\gamma)\pi(1_{E})\xi^{(n)}_{j}-\pi(1_{E})\xi^{(n)}_{j}\|^{2}=\lim_{n\to\infty}\int_{X}\sum_{\gamma\in F}\|\pi(x,\gamma^{-1}(x))\xi_{\gamma^{-1}(x)}^{(n)}-\xi^{(n)}_{x}\|^{2}\,d\mu(x).\]
Since $\|\xi^{(n)}_{x}\|=1$ for almost every $x\in X$, it follows that $\sum_{\gamma\in F}\|\pi(x,\gamma^{-1}(x))\xi_{\gamma^{-1}(x)}^{(n)}-\xi^{(n)}_{x}\|^{2}\leq 4|F|$. Hence the dominated convergence theorem implies that 
\[\lim_{n\to\infty}\sum_{j=1}^{\infty}\sum_{\gamma\in F}\|\pi(\gamma)(\pi(1_{E})\xi^{(n)}_{j}-\pi(1_{E})\xi^{(n)}_{j}\|^{2}=0.\]
However, 
\[\sum_{j=1}^{\infty}\|(\pi(1_{E})\xi^{(n)}_{j}\|^{2}=\int_{E}\|\xi^{(n)}_{x}\|^{2}\,d\mu(x)=\mu(E).\]
Since $\mu(E)>0$, for all sufficiently large $n$ there is some $j\in \N$ with
\[\sum_{\gamma\in F}\|\pi(\gamma)(\pi(1_{E})\xi^{(n)}_{j}-\pi(1_{E})\xi^{(n)}_{j}\|^{2}<\varepsilon\|\pi(1_{E})\xi^{(n)}_{j}\|^{2}.\]
Setting $\zeta=\frac{\pi(1_{E})\xi^{(n)}_{j}}{\|\pi(1_{E})\xi^{(n)}_{j}\|}$, we have that 
\[\|\pi(\gamma)\zeta-\zeta\|<\varepsilon, \text{ for all $\gamma\in F$}.\]
Since $F,\varepsilon$ were arbitrary, this completes the proof.

(\ref{item: a.i. measurable sets}) implies (\ref{item: weak containment}):
Let $C=\overline{\{\varphi_{\xi}:\xi\in \cH,\|\xi\|=1\}}^{wk^{*}}\subseteq C^{*}(\cR)^{*}$. Since $\|\varphi_{\xi}\|\leq 1$ if $\|\xi\|\leq 1,$ the Banach-Alaoglu theorem implies that $C$ is weak$^{*}$-compact. It suffices to show that $\psi\in C$. 
We start with the following claim:

\emph{Claim: given $\cR$-invariant measurable sets $E_{1},\cdots,E_{n}$, and $\gamma_{1},\cdots,\gamma_{n}\in [\cR]$, there is a $\varphi\in C$ with $\varphi(\gamma_{j})=1$, $\varphi(1_{E_{i}})=\mu(E_{i})$ for all $1\leq i,j\leq n$. }
If the claim is true, then by weak$^{*}$-compactness of $C$, there is a 
\[\varphi_{0}\in \bigcap_{\gamma\in [\cR], E\subseteq X \text{$\cR$-invariant and measurable}}\{\varphi\in C:\varphi(\gamma)=1,\varphi(1_{E})=\mu(E)\}.\]
By Lemma \ref{lem: inv set reduction}, it follows that $\varphi_{0}=\psi$. Hence it suffices to prove the claim.

To prove the claim, fix $\cR$-invariant measurable sets $E_{1},\cdots,E_{n}$, and $\gamma_{1},\cdots,\gamma_{n}\in [\cR]$. Let $\cF$ be the $\sigma$-algebra generated by $(E_{j})_{j=1}^{n}$ and find pairwise disjoint, $\cR$-invariant measurable sets $A_{1},\cdots,A_{k}$ so that $\mu(A_{i})>0$ and so that we have that $\cF$ agrees modulo null sets with $\left\{\bigsqcup_{j\in D}A_{j}:D\subseteq \{1,\cdots,k\}\right\}.$ By linearity, to prove the claim for $E_{1},\cdots,E_{n}$, and $\gamma_{1},\cdots,\gamma_{n}$, it suffices to find $\varphi\in C$ with $\varphi(1_{A_{i}})=\mu(A_{i})$, $\varphi(\gamma_{j})=1$ for $1\leq i\leq k,1\leq j\leq n$. By  (\ref{item: a.i. measurable sets}) applied to each $A_{i}$, for $i=1,\cdots,k$ and $l\in \N$, we may find $\xi^{(i)}_{l}\in \pi(1_{A_{i}})\cH$ with $\|\xi^{(i)}_{l}\|_{2}=1$ and $\|\pi(\gamma_{j})\xi^{(i)}_{l}-\xi^{(i)}_{l}\|\to_{l\to\infty}0$ for $j=1,\cdots,n$. Since each $\pi(1_{A_{i}})\cH$ is orthogonal, if we set
\[\xi_{l}=\sum_{i=1}^{k}\sqrt{\mu(A_{i})}\|\xi^{(i)}_{l}\|,\]
then $\|\xi_{l}\|=1$ and $\|\pi(\gamma_{j})\xi_{l}-\xi_{l}\|\to_{l\to\infty}0$ for $j=1,\cdots,n$. Note that $\varphi_{\xi_{l}^{(j)}}(1_{A_{i}})=\delta_{j=i}$ for all $i,j=1,\cdots,k$, so $\varphi_{\xi_{l}}(1_{A_{i}})=\mu(A_{i})$ for all $i=1,\cdots,k$. By Cauchy-Schwartz we have that $\varphi_{\xi_{l}}(\gamma_{j})\to_{l\to\infty}1$  for all $j=1,\cdots,n$. Thus, any 
\[\varphi\in \bigcap_{l=1}^{\infty}\overline{\{\varphi_{\xi_{m}}:m\geq l\}}^{wk^{*}}\]
verifies the claim. 

\end{proof}

\subsection{Strong ergodicity and weak containment}

In this section, we show that we may characterize strong ergodicity in terms of weak containment of representations. In particular, we show that an ergodic probability measure-preserving relation $\cR$ is strongly ergodic if and only every unitary representation which is weakly equivalent to $\alpha$ in fact contains $\alpha$. This should be regarded as analogous to \cite[Proposition 3.2]{BMOFull}. In light of the preceding section and the fact that $\alpha$ is the analogue of the trivial representation, this allows one to view strong ergodicity as a weak version of Property (T).

To show that (when $\cR$ is strongly ergodic) any representation weakly equivalent to $\alpha$ must contain it, it is helpful to use the following well known fact. 

\begin{lem}\label{lem: its all spectral gap}
 Let $\cR$ be a strongly ergodic, probability measure-preserving, discrete equivalence relation. Then $P_{\C1}\in \alpha(C^{*}(\cR))$. 
\end{lem}

\begin{proof}
Since $\cR$ is strongly ergodic, by \cite[Theorem 2.4]{SchmidtCohom} we may find $\gamma_{1},\cdots,\gamma_{n}\in [\cR]$ with $\gamma_{1}=\id$, and
\[\frac{1}{2n}\left\|\left(\sum_{j=1}^{n}\alpha_{\gamma_{j}}+\alpha_{\gamma_{j}^{-1}}\right)\big|_{L^{2}_{0}(X)}\right\|<1.\]
Set $T=\frac{1}{2n}\sum_{j=1}^{n}\alpha_{\gamma_{j}}+\alpha_{\gamma_{j}^{-1}}$. The above then shows that $\|P_{\C 1}-T^{n}\|\to_{n\to\infty}0$. 

\end{proof}

A special case of interest for unitary representations of equivalence relations is the case where the field of Hilbert spaces are fiberwise one-dimensional. In this case, a unitary representation corresponds to a map $c\colon \cR\to S^{1}$ which satisfies $c(x,y)c(y,z)=c(x,z)$ for almost every $x\in X$ and all $y,z\in [x]_{\cR}$. These maps are typically called \emph{cocycles} in the literature. A cocycle is a \emph{coboundary} if there is a measurable $u\colon X\to S^{1}$ with $c(x,y)=u(x)\overline{u(y)}$. We identify two cocycles if they agree modulo nulls sets. We equip the space of (equivalence classes of) cocycles with the topology of local convergence in measure. So $c_{n}\to c$ if for every finite measure $F\subseteq \cR$ and every $\varepsilon>0$, we have that 
\[\mu_{\cR}(\{(x,y)\in F:|c_{n}(x,y)-c(x,y)|>\varepsilon\})\to_{n\to\infty}0.\]
It is direct to check that this corresponds to convergence with respect to the metric 
\[d(c_{1},c_{2})=\sum_{n=1}^{\infty}\frac{2^{-n}}{\mu_{\cR}(F_{n})+1}\int_{F_{n}}|c_{1}(x,y)-c_{2}(x,y)|\,d\mu_{\cR}(x,y)\]
where $(F_{n})_{n=1}^{\infty}$ is any sequence of finite measure sets with $\cR=\bigcup_{n}F_{n}$ modulo null sets. A cocycle is said to be an \emph{approximate coboundary} if it is a limit, in the topology of local convergence in measure, of coboundaries. For ergodic relations, \cite[Proposition 2.3]{SchmidtCohom} shows that strong ergodicity is equivalent to every approximate coboundary being a coboundary. Using this allows us to characterize strong ergodicity in terms of weak containment of representations, following the argument in \cite[Proposition 3.2]{BMOFull}.

\begin{cor}\label{cor: weak containment strong ergodic}
Suppose that $\cR$ is an ergodic,probability measure-preserving, discrete equivalence relation.  The following are equivalent:
\begin{enumerate}[(i)]
\item $\cR$ is strongly ergodic, \label{item: R is SE}
\item \label{item: R is SE in terms of weak equivalence} if $((\cH_{x})_{x\in X})$ is a unitary representation of $\cR$ and $L^{2}(X)$ is weakly equivalent to $\cH:=\int_{X}^{\oplus}\cH_{x}\,d\mu(x)$ as representations of $\cR$, then $L^{2}(X)$ embeds into $\cH$ as a representation of $\cR$.
\end{enumerate}
\end{cor}

\begin{proof}
(\ref{item: R is SE}) implies (\ref{item: R is SE in terms of weak equivalence}):
Since $\cH$ is weakly equivalent to $L^{2}(X)$, there is unique isometric $*$-homomorphism $\Phi\colon \alpha(C^{*}(\cR))\to \pi(C^{*}(\cR))$ so that $\pi=\Phi\circ \alpha$. By Lemma \ref{lem: its all spectral gap}, we have that $P_{\C 1}\in \alpha(C^{*}(\cR))$. Since $\Phi$ is an isometric $*$-homomorphism, $q=\Phi(P_{\C 1})$ is a nonzero projection which satisfies that $q\pi(f)\pi(\gamma)q=\int f\,d\mu$ for all $f\in L^{\infty}(X),\gamma\in [\cR]$. Since $q\ne 0$, we can find a unit vector $\zeta\in \cH$ with $q\zeta=\zeta$. Then $\ip{\pi(f)\pi(\gamma)\zeta,\zeta}=\ip{q\pi(f)\pi(\gamma)q\zeta,\zeta}=\int f\,d\mu$. Since $\ip{f\alpha(\gamma)1,1}=\int f\,d\mu$ for all $f\in L^{\infty}(X),\gamma\in [\cR]$, it follows that there is a unique  isometry $V\in B(L^{2}(X),\cH)$ satisfying that $V1=\zeta$ and $V(\alpha(f)\alpha(\gamma))=\pi(f)\pi(\gamma)V$ for all $f\in L^{\infty}(X),\gamma\in [\cR]$.

(\ref{item: R is SE in terms of weak equivalence}) implies (\ref{item: R is SE}): Suppose $\cR$ is not strongly ergodic. Then, by \cite[Proposition 2.3]{SchmidtCohom} we may find a unitary representation $c\colon \cR\to S^{1}$  which is not a coboundary, but is a limit of coboundaries. Set $\cH_{x}=\C 1$ and for $(x,y)\in \cR$ define $\pi_{c}(x,y)\in \cU(\C)=c(x,y)$. Set $\cH=\int_{X}^{\oplus}\cH_{x}\,d\mu(x)=L^{2}(X)$. 

We claim that $\pi_{c}$ is weakly equivalent to the unitary representation $\alpha$ on $L^{2}(X)$, but does not contain $\alpha$.
 To see that $\pi_{c}$ is weakly equivalent to $\alpha$, let $v_{n}\colon X\to S^{1}$ be a sequence of measurable functions so that $(x,y)\mapsto v_{n}(x)\overline{v_{n}(y)}$ converges to $c$ locally in measure. Note that a $\xi\in \Meas((\cH_{x})_{x\in X})$ may be thought of as a measurable function $\xi\colon X\to\C$. Then, for any finite measures $Q\subseteq \cR$ we have for all large $n$ that 
 \[\mu_{\cR}(\{(x,y)\in Q:|c(x,y)\xi(y)\overline{\xi(x)}-\overline{v_{n}(y)}\xi(y)\overline{\xi(x)}v_{n}(x)|>\varepsilon\})<\varepsilon,\]
 since $(x,y)\mapsto v_{n}(x)\overline{v_{n}(y)}$ converges to $c$. Since $\overline{v_{n}(y)}\xi(y)\overline{\xi(x)}v_{n}(x)=\ip{(\alpha(x,y)\overline{v_{n}}(y)\xi(y),\overline{v_{n}(x)}\xi(x)}_{\C}$, and $c(x,y)\xi(y)\overline{\xi(x)}=\ip{c(x,y)\xi(y),\xi(x)}_{\C}$, this proves that $\pi_{c}$ is weakly contained in $\alpha$. The same argument works to show that $\pi_{c}$ weakly contains $\alpha$. Hence the two unitary representations are weakly equivalent.

Suppose $\alpha$ embeds into $\pi_{c}$. Then we may find an isometry $V\colon L^{2}(X)\to L^{2}(X)$ so that $V(\alpha(f\gamma))=\pi_{c}(f\gamma)V$ for all $f\in L^{\infty}(X)$ and $\gamma\in [\cR]$. Set $\xi=V(1)$, then for all $\gamma\in [\cR]$ we have $c(x,\gamma(x))\xi(\gamma(x))=\xi(x)$ for almost every $x\in X$. Applying this to $\gamma$ in a countable subgroup which generates $\cR$, we see that $c(x,y)\xi(y)=\xi(x)$ for almost every $(x,y)\in \cR$. In particular, $|\xi|$ is $\cR$-invariant modulo null sets. By ergodicity of $\cR$ and the fact that $\|\xi\|_{2}=1$, we see that $\xi$ is valued in $S^{1}$ modulo null sets. Thus $c(x,y)=\xi(x)\overline{\xi(y)}$ for almost every $(x,y)\in \cR$ and this forces $c$ to be a coboundary, a contradiction. 
\end{proof}

\section{Proof of the Main Theorem}

\subsection{Background on tracial von Neumann algebras and ultrapowers} \label{sec: ultrapower}

We need the general notion of a tracial von Neumann algebra.

\begin{defn}
Let $\cH$ be a Hilbert space. We say that $M\subseteq B(\cH)$ is a \emph{von Neumann algebra} if $M$ is a $*$-subalgebra which contains the identity and is closed in the strong operator topology.
\end{defn}

Note that by \cite[Corollary IX.5.2]{Conway}, it is equivalent to require that $M$ is closed in the weak operator topology.
\begin{defn}
A \emph{tracial von Neumann algebra} is a pair $(M,\tau)$ where $\tau\colon M\to \C$ is a faithful state which is in addition:
\begin{enumerate}[(i)]
    \item \emph{tracial}, i.e. $\tau(xy)=\tau(yx)$ for all $x,y\in M$,
    \item \emph{normal}, i.e. $\tau|_{\{x\in M:\|x\|\leq 1\}}$ is WOT-continuous.
\end{enumerate}
\end{defn}

Here are some examples.

\begin{example}
For $n\in \N$, $(\M_{n}(\C),\tr)$ is a tracial von Neumann algebra, where $\tr$ is the normalized trace
\[\tr(A)=\frac{1}{n}\sum_{j=1}^{n}A_{jj}.\]   
\end{example}

\begin{example}\label{ex: abelian example}
Let $(X,\mu)$ be a probability space. View $L^{\infty}(X,\mu)\subseteq B(L^{2}(X,\mu))$ via identifying $f$ with the operator $(\xi\mapsto f\xi)$. Then the weak-operator topology on $L^{\infty}(X,\mu)$ coincides with the weak$^{*}$-topology on $L^{\infty}(X,\mu)$ as the dual of $L^{1}(X,\mu)$. Thus $(L^{\infty}(X,\mu),\int \cdot\,d\mu)$ is a tracial von Neumann algebra.
    
\end{example}

\begin{example}
Let $\cR$ be a discrete, probability measure-preserving, equivalence relation on a standard probability space $(X,\mu)$. We set $L(\cR)=\overline{\lambda(C^{*}(\cR))}^{SOT}$. And $\tau(x)=\ip{x1_{\Delta},1_{\Delta}}$ where $\Delta=\{(x,x):x\in X\}$. By \cite[Proposition 2.8]{FelMoore}, this is a tracial von Neumann algebra.
\end{example}

Motivated by Example \ref{ex: abelian example} for a tracial von Neumann algebra $(M,\tau)$ we define $\|\cdot\|_{2}$ on $M$ by $\|x\|_{2}=\tau(x^{*}x)^{1/2}$. 

For a von Neumann algebra $M$, we let $M_{s.a.}=\{x\in M:x=x^{*}\}, \cU(M)=\{u\in M:u^{*}u=uu^{*}=1\}$. Note that for the case $M=L^{\infty}(X,\mu)$ these coincide with real-valued and $S^{1}$-valued functions, respectively. If $M\subseteq B(\cH)$, we set $M'=\{T\in B(\cH):Ta=aT \textnormal{ for all $a\in M$}\}$. 

A particular construction that will be relevant for us is the ultraproduct of a sequence of von Neumann algebras. A \emph{free ultrafilter} on $\N$ is, by definition, a nonzero homomorphism $\omega\colon \ell^{\infty}(\N)/c_{0}(\N)\to \C$. Note that such homomorphisms exist in abundance by \cite[Theorems VIII.2.1 and VIII.4.6]{Conway}.  We typically write $\lim_{n\to\omega}a_{n}$ instead of $\omega((a_{n})_{n=1}^{\infty}+c_{0}(\N))$ when $(a_{n})_{n=1}^{\infty}\in \ell^{\infty}(\N)$. Since $\omega$ is a nonzero homomorphism, we have that $\lim_{n\to\omega}1=1$ and since $\lim_{n\to\omega}a_{n}=0$ if $(a_{n})_{n=1}^{\infty}\in c_{0}(\N)$, we see that $\lim_{n\to\omega}a_{n}=\lim_{n\to\infty}a_{n}$ if $a_{n}$ converges. Moreover, we have
\begin{itemize}
    \item $\lim_{n\to\omega}\overline{a_{n}}=\overline{\lim_{n\to\omega}a_{n}}$, for all $(a_{n})_{n=1}^{\infty}\in \ell^{\infty}(\N)$ (by \cite[Proposition VIII.1.12(a)]{Conway}),
    \item $\lim_{n\to\omega}a_{n}b_{n}=(\lim_{n\to\omega}a_{n})(\lim_{n\to\omega}b_{n})$, $\lim_{n\to\omega}a_{n}+b_{n}=\lim_{n\to\omega}a_{n}+\lim_{n\to\omega}b_{n}$  for all $(a_{n})_{n=1}^{\infty},(b_{n})_{n=1}^{\infty}\in \ell^{\infty}(\N)$ (since $\omega$ is a homomorphism),
    \item $\lim_{n\to\omega}a_{n}\leq \lim_{n\to\omega}b_{n}$ if $(a_{n})_{n=1}^{\infty},(b_{n})_{n=1}^{\infty}\in \ell^{\infty}(\N,\R)$ and $a_{n}\leq b_{n}$ for all $n$ (by \cite[Proposition VIII.1.12(c)]{Conway}). 
\end{itemize}

Let $(M_{n},\tau_{n})$ be a sequence of tracial von Neumann algebras and $\omega$ a free ultrafilter on $\N$. We define their \emph{ultraproduct} $(\prod_{n\to\omega}M_{n},\tau_{\omega})$ by 
\[\prod_{n\to\omega}M_{n}=\frac{\{(a_{n})_{n=1}^{\infty}\in \prod_{n}M_{n}:\sup_{n}\|a_{n}\|<+\infty\}}{\{(a_{n})_{n=1}^{\infty}\in \prod_{n}M_{n}:\sup_{n}\|a_{n}\|<+\infty,\lim_{n\to\omega}\|a_{n}\|_{2}=0\}}\]
and $\tau_{\omega}((a_{n})_{n\to\omega})=\lim_{n\to\omega}\tau_{n}(a_{n})$, where $(a_{n})_{n\to\omega}$ is the image in $\prod_{n\to\omega}M_{n}$ of a sequence $(a_{n})_{n=1}^{\infty}\in \prod_{n}M_{n}$ with $\sup_{n}\|a_{n}\|<+\infty$. By the proof of \cite[Lemma A.9]{BO}, we have that $(\prod_{n\to\omega}M_{n},\tau_{\omega})$ is a tracial von Neumann algebra.

If each $(M_{n},\tau_{n})$ is a fixed algebra $(M,\tau)$, we will use $M^{\omega}$ instead of $\prod_{n\to\omega}M$. In particular, for a probability space $(X,\mu)$, we use $L^{\infty}(X,\mu)^{\omega}$ for the ultrapower with respect to the trace $\int_{X}\cdot \,d\mu$. If $\mu$ is understood from context, we often denote this by $L^{\infty}(X)^{\omega}$. 

If $\phi\in \Aut(X,\mu)$ then we have an induced action $\alpha_{\phi}^{\omega}$ on $L^{\infty}(X,\mu)^{\omega}$ by 
\[\alpha_{\phi}^{\omega}((f_{n})_{n\to\omega})=(\alpha_{\phi}(f_{n}))_{n\to\omega}.\]
We often abuse notation and use $\alpha_{\phi}$ instead of $\alpha_{\phi}^{\omega}$ if it is clear from context. If $(X,\mu)$ is now standard and $\cR$ is a discrete, probability measure-preserving relation on $(X,\mu)$, we set 
\[\Fix_{\cR}(L^{\infty}(X)^{\omega})=\{f\in L^{\infty}(X)^{\omega}:\alpha_{\gamma}(f)=f \textnormal{ for all $\gamma\in [\cR]$}\}.\]
This is relevant for strong ergodicity, since by \cite[Proposition 1.1.3]{ConnesSE} $\cR$ is strongly ergodic if and only if $\Fix_{\cR}(L^{\infty}(X)^{\omega})=\C1$.

\subsection{Proof of the main theorem}
Let $\cR$ be discrete, probability measure-preserving equivalence relation on a standard probability space $(X,\mu)$. 
We define the equivalence relation von Neumann algebra of $\cR$ to be $L(\cR)=\overline{\lambda(C^{*}(\cR))}^{WOT}.$ 
Note that if $\cS\leq \cR$ then we have a natural injective $*$-homomorphism $\iota \colon\C(\cS)\to \C(\cR)$ given by extension by zero on $\cR\setminus \cS$. 
Thus we may view $\C(\cS)\subseteq \C(\cR)$.  
It is direct to check that \cite[Proposition 7.1.15]{KadisonRingroseII},
\[\ip{\lambda_{\cR}(\iota(k))1_{\Delta},1_{\Delta}}_{L^{2}(\cR)}=\ip{\lambda_{\cS}(k)1_{\Delta},1_{\Delta}}_{L^{2}(\cS)},\]
i.e. $\iota$ is a trace-preserving. It is a folklore fact (e.g. this is the same argument as in the discussion of Section 2 of \cite{BekkaOAsuperrigid}, see also \cite[Proposition 2.10]{AW25}) that this implies that $\iota$ extends uniquely to trace-preserving $*$-homomorphism $L(\cS)\hookrightarrow L(\cR)$, which we still denote by $\iota$. By \cite[Proposition III.5.3]{Taka} the fact that $\iota$ is trace-preserving implies that it is SOT-SOT continuous when restricted to the unit ball, by \cite[Proposition 7.1.15]{KadisonRingroseII} this also implies that $\iota$ is WOT-WOT when restricted to the unit ball. By the preceding discussion, we may identify $L(\cS)$ with a von Neumann subalgebra of $L(\cR)$. Let $V\colon L^{2}(\cS)\to L^{2}(\cR)$ be the isometry $(V\xi)(x,y)=1_{\cS}(x,y)\xi(x,y)$. And define $\Phi\colon B(L^{2}(\cS))\to B(L^{2}(\cR))$ by $\Phi(T)=V^{*}TV$. Then $\Phi(T)=T$ if $T\in L(\cS)$ (viewing $L(\cS)$ inside  $L(\cR)$) and 
\[\Phi(f\lambda_{\cR}(\gamma))=f1_{\{x:[\gamma^{-1}(x)]_{\cS}=[x]_{\cS}\}}\lambda_{\cR}(\gamma), \textnormal{ for all $f\in L^{\infty}(X),\gamma\in[\cR]$.}\]
Since $\Phi$ is WOT-WOT continuous, and 
\[L(\cR)=\overline{\Span\{f\lambda_{\cR}(\gamma):f\in L^{\infty}(X),\gamma\in[\cR]\}}^{WOT},\]
it follows that $\Phi(L(\cR))\subseteq L(\cS)$. We set $\E_{L(\cS)}=\Phi|_{L(\cR)}$.

Our ultimate goal in this section is to prove Theorem \ref{thm: I said the real intro thm}. Following \cite[Section 3]{BMOFull}, our strategy will hinge on working with a relation $\widehat{\cS}$ containing $\cS$ satisfying that $\Fix_{\widehat{\cS}}(L^{\infty}(X)^{\omega})'\cap L(\cR)=L(\widehat{\cS})$. It will be helpful (both on a technical and a conceptual level) to rephrase this condition in several equivalent ways.

\begin{prop}\label{prop: TFAE commutant condition}
 Let $\cS\leq \cR$ be discrete, probability measure-preserving equivalence relations on a standard probability space $(X,\mu)$. Fix $\omega\in \beta\N\setminus \N$. The following are equivalent:
 \begin{enumerate}[(i)]
 \item $\Fix_{\cS}(L^{\infty}(X)^{\omega})'\cap L(\cR)=L(\cS)$, \label{item: asymptotic bicentralizer}
 \item $[\cS]=\{\gamma\in [\cR]:\alpha_{\gamma}(f)=f \textnormal{ for all $f\in \Fix_{\cS}(L^{\infty}(X)^{\omega})$}\}$, \label{item: asymptotic bifixing}
 \item if $\gamma\in [\cR]$, $f\in L^{\infty}(X)$ satisfy that $\alpha_{\gamma}(k)f=kf$ for every $k\in \Fix_{\cS}(L^{\infty}(X)^{\omega})$, then 
 \[\mu(\{x:f(x)\ne 0\}\setminus\{x:[\gamma^{-1}(x)]_{\cS}=[x]_{\cS}\})=0.\]
 \label{item:slightly messier condition}
 \end{enumerate}
\end{prop}

\begin{proof}

(\ref{item: asymptotic bicentralizer}) implies (\ref{item:slightly messier condition}):Suppose $f,\gamma$ satisfy the assumptions of 
(\ref{item:slightly messier condition}), and set $u_{\gamma}=\lambda_{\cR}(\gamma)$. Then $fu_{\gamma}$ commutes with $k$ for all $k\in \Fix_{\cS}(L^{\infty}(X)^{\omega})$. Thus by (\ref{item: asymptotic bicentralizer}), it follows that $fu_{\gamma}\in L(\cS)$. Hence
\[fu_{\gamma}=\E_{L(\cS)}(fu_{\gamma})=f1_{\{x:[\gamma^{-1}(x)]_{\cS}=[x]_{\cS}\}}u_{\gamma}.\]
Since $u_{\gamma}$ is a unitary, it follows that $f=f1_{\{x:[\gamma^{-1}(x)]_{\cS}=[x]_{\cS}\}}$ almost everywhere.

(\ref{item:slightly messier condition}) implies (\ref{item: asymptotic bifixing}): Suppose $\alpha_{\gamma}(k)=k$ for all $k\in L^{\infty}(X)^{\omega}$. Then  $\gamma$ satisfies (\ref{item:slightly messier condition}) with $f=1$. The conclusion  of (\ref{item:slightly messier condition}) then says that $\gamma\in [\cS]$. So $[\cS]\supseteq \{\gamma\in [\cR]:\alpha_{\gamma}(f)=f \textnormal{ for all $f\in \Fix_{\cS}(L^{\infty}(X)^{\omega})$}\}$. The reverse inclusion is an exercise.

(\ref{item: asymptotic bifixing}) implies (\ref{item: asymptotic bicentralizer}): By \cite[Proposition 6.1]{Dye}, there is a probability measure-preserving relation $\widehat{\cS}$ with $\cS\leq \widehat{\cS}\leq \cR$ so that 
\[L(\widehat{\cS})=\Fix_{\cS}(L^{\infty}(X)^{\omega})'\cap L(\cR).\]
Since $u_{\gamma}fu_{\gamma}^{-1}=\alpha_{\gamma}(f)$ for all $f\in L^{\infty}(X)^{\omega},\gamma\in [\cR]$, it follows that  
\[[\widehat{\cS}]=\{\gamma\in [\cR]:\alpha_{\gamma}(f)=f \textnormal{ for all $f\in \Fix_{\cS}(L^{\infty}(X)^{\omega})$}\}.\]
Our hypothesis thus says that $[\widehat{\cS}]=[\cS]$, and thus $\widehat{\cS}=\cS$ modulo null sets. So $L(\widehat{\cS})=L(\cS)$.

\end{proof}

The following Lemma is the analogue of \cite[Lemma 3.4]{BMOFull}

\begin{lem}\label{lem: analogue of BMO}
Let $\cS\leq \cR$ be discrete, probability measure-preserving equivalence relations on a standard probability space $(X,\mu)$. Suppose that $\cS$ satisfies 
\[L(\cS)=\Fix_{\cS}(L^{\infty}(X)^{\omega})'\cap L(R).\]
Then $L^{2}(\cR/\cS)\preceq L^{2}(X)$ as representations of $C^{*}(\cR)$. 

\end{lem}

\begin{proof}

\emph{Claim 1: Let $\gamma=(\gamma_{1},\cdots,\gamma_{n})\in [\cR]$. be given. For $v\in \Meas(X,S^{1})$, set $v_{\gamma}=(\overline{v}\alpha_{\gamma_{j}}(v))_{j=1}^{n}$. Then}
\[(1_{\{x:[\gamma_{j}^{-1}(x)]_{\cS}=[x]_{\cS}\}})_{j=1}^{n}\in \overline{co}^{wk^{*}}\{v_{\gamma}:v\in \Meas(X,S^{1})\}.\]

Let us explain why Claim 1 implies the Lemma. Recall that $L^{2}(\cR/\cS)$ as a representation of $C^{*}(\cR)$ has $\xi_{\cR/\cS}$ as a cyclic vector, where $\xi_{\cR/\cS}(x,c)=1_{x\in c}$.  Given $f_{1},\cdots,f_{n}\in L^{\infty}(X)$ and $\gamma_{1},\cdots,\gamma_{n}\in [\cR]$, we have that:
\[\ip{f_{j}\lambda_{\cR/\cS}(\gamma_{j})\xi_{\cR/\cS},\xi_{\cR/\cS}}=\int f_{j}1_{\{x:[\gamma_{j}^{-1}(x)]_{\cS}=[x]_{\cS}\}}\,d\mu.\]
Let $\varepsilon>0$. By claim $1$,  we may find $v_{1},\cdots,v_{k}\in \Meas(X,S^{1})^{n}$ and $\lambda_{1},\cdots,\lambda_{k}\in [0,1]$ with $\sum_{j=1}^{k}\lambda_{j}=1$ and
\[\left|\int f_{j}1_{\{x:[\gamma_{j}^{-1}(x)]_{\cS}=[x]_{\cS}\}}\,d\mu-\sum_{i=1}^{k}\lambda_{i}\int f_{j}\overline{v_{i}}\alpha_{\gamma_{j}}(v_{i})\,d\mu\right|<\varepsilon, \textnormal{ for all $j=1,\cdots,n$}.\]
I.e.
\[\left|\ip{f_{j}\lambda_{\cR/\cS}(\gamma_{j})\xi_{\cR/\cS},\xi_{\cR/\cS}}-\sum_{i=1}^{k}\lambda_{j}\ip{f_{j}\alpha_{\gamma_{j}}(v_{i}),v_{i}}\right|<\varepsilon.\]
Since $f_{1},\cdots,f_{n}\in L^{\infty}(X)^{n},\gamma_{1},\cdots,\gamma_{n}\in [\cR]$ are arbitrary, this shows that $L^{2}(\cR/\cS)\preceq L^{2}(X)$ as representations of $C^{*}(\cR)$.

We now proceed to the proof of Claim 1. Fix $\gamma_{1},\cdots,\gamma_{n}\in [\cR]$. 
For finite $F\subseteq [\cS]$ and $\varepsilon>0$, we let 
\[E_{F,\varepsilon}=\bigcap_{\sigma\in F}\{v\in \Meas(X,S^{1}):\|\alpha_{\sigma}(v)-v\|_{2}<\varepsilon\}.\]
We then set
\[C_{F,\varepsilon}=\overline{\co}^{wk^{*}}\{v_{\gamma}:v\in E_{F,\varepsilon}\}\subseteq L^{\infty}(X)^{n},\]
and set
\[C=\bigcap_{F,\varepsilon}C_{F,\varepsilon}.\]
So $C$ is a nonempty, compact, convex set in $L^{\infty}(X)^{n}$. For later use, it will be help to note the following.

\emph{Claim 2. For all $(f_{1},\cdots,f_{n})\in C$, we have that $f_{j}1_{\{x:[\gamma_{j}^{-1}(x)]_{\cS}=[x]_{\cS}\}}=1_{\{x:[\gamma_{j}^{-1}(x)]_{\cS}=[x]_{\cS}\}}$ almost everywhere for all $j=1,\cdots,n$}
To see this, note that we may find $\sigma_{1},\cdots,\sigma_{n}\in [\cS]$ so that 
\[\sigma_{j}^{-1}|_{\{x:[\gamma_{j}^{-1}(x)]_{\cS}=[x]_{\cS}\}}=\gamma_{j}^{-1}|_{\{x:[\gamma_{j}^{_1}(x)]_{\cS}=[x]_{\cS}\}}, \text{ for all $j=1,\cdots,n$}.\]
Set $F=\{\sigma_{1},\cdots,\sigma_{n}\}$. Then for every $v\in E_{F,\varepsilon}$ we have that 
\[\|\overline{v}\alpha_{\gamma_{j}}(v)1_{\{x:[\gamma_{j}^{-1}(x)]_{\cS}=[x]_{\cS}\}}-1_{\{x:\gamma_{j}^{-1}(x)]_{\cS}=[x]_{\cS}\}}\|_{2}<\varepsilon \text{ for all $j=1,\cdots,n$.}\]
By weak$^{*}$-semicontinuity of $\|\cdot\|_{2}$ and the triangle inequality, we have that $\{f\in L^{\infty}(X):\|f\|_{2}\leq\varepsilon\}$ is weak$^{*}$-closed and convex. Hence, we have that $\|f_{j}1_{\{x:\gamma_{j}^{-1}(x)]_{\cS}=[x]_{\cS}}-1_{\{x:\gamma_{j}^{-1}(x)]_{\cS}=[x]_{\cS}}\|_{2}\leq \varepsilon$ for all $(f_{1},\cdots,f_{n})\in C_{F,\varepsilon}$. In particular,
\[\|f_{j}1_{\{x:\gamma_{j}^{-1}(x)]_{\cS}=[x]_{\cS}\}}-1_{\{x:\gamma_{j}^{-1}(x)]_{\cS}=[x]_{\cS}\}}\|_{2}\leq \varepsilon, \text{ for all $(f_{1},\cdots,f_{n})\in C$ and all $j=1,\cdots,n$}.\]
Since $\varepsilon>0$ was arbitrary, this proves Claim 2.

We now finish the proof of Claim 1. Since $C$ is weak$^{*}$-closed and bounded as a subset of $L^{\infty}(X)^{n}$, it is also $\|\cdot\|_{2}$-closed as a subset of $L^{2}(X)^{n}$. Since $C$ is then closed and convex as a subset of $L^{2}(X)^{n}$, we can let $k=(k_{1},\cdots,k_{n})\in C$ be an element with minimal $\|\cdot\|_{2}$ in $C$. Given $v\in \cU(L^{\infty}(X)^{\omega}\cap L(\cS)')$, set $\widetilde{y}=(\alpha_{\gamma_{j}}(v)\overline{v}k_{j})_{j=1}^{n}$, and write $v=(v_{n})_{n\to\omega}$ with $v_{n}\in \Meas(X,S^{1})$. Set
\[y_{j}=wk^{*}-\lim_{n\to\omega}\alpha_{\gamma_{j}}(v_{n})\overline{v_{n}}k_{j},\]
i.e. $y_{j}$ is the unique element of $L^{\infty}(X)$ satisfying
\[\int y_{j}h\,d\mu=\lim_{n\to\omega}\int \alpha_{\gamma_{j}}(v_{n})\overline{v_{n}}k_{j}h\,d\mu \text{ for all $h\in L^{1}(X)$}.\]
Finally, set $y=(y_{1},\cdots,y_{n})$, and note that
\[\|y-\widetilde{y}\|_{2}^{2}=\|\widetilde{y}\|_{2}^{2}-\|y\|_{2}^{2}.\]
Then $y\in C$, and $\|y\|_{2}\leq \|\widetilde{y}\|_{2}=\|k\|_{2}$. By minimality $y=k$ for all $j=1,\cdots,n$. This forces $\|\widetilde{y}\|_{2}=\|y\|_{2}$, and by the above we conclude that $\widetilde{y}=y$. 
So we have shown that
\[\alpha_{\gamma_{j}}(v)k_{j}=vk_{j}, \textnormal{for all $j=1,\cdots,n$ and all $v\in \cU(\Fix_{\cS}(L^{\infty}(X)^{\omega}))$}.\]
Since every $f\in \Fix_{\cS}(L^{\infty}(X))^{\omega})$ is a linear combination of four elements of $\cU(\Fix_{\cS}(L^{\infty}(X)^{\omega}))$, it follows that 
\[\alpha_{\gamma_{j}}(f)k_{j}=fk_{j}, \textnormal{for all $j=1,\cdots,n$ and all $f\in \Fix_{\cS}(L^{\infty}(X)^{\omega})$}.\]
Hence by Proposition \ref{prop: TFAE commutant condition} it follows that $k_{j}=k_{j}1_{\{x:[\gamma_{j}^{-1}(x)]_{\cS}=[x]_{\cS}\}}$ almost everywhere. Claim 2 now completes the proof of Claim 1.

\end{proof}

If $\cS\leq \cR$ is an inclusion of probability measure-preserving orbit equivalence relations, we will need to characterize when $L^{2}(\cR/\cS)$ has a $[\cR]$-invariant vector. If $\cR$ is ergodic, note that $x\mapsto |[x]_{\cR}/\cS|$ is an $\cR$-invariant measurable function. Hence it is almost surely constant. The \emph{index} of $\cS$ in $\cR$ is, by definition, the almost sure value of this function. We denote the index by $[\cR:\cS]$. 

\begin{prop}\label{prop:translation of invariant vectors}
Suppose that $\cS\leq \cR$ are discrete, measure-preserving equivalence relations on a standard probability space $(X,\mu)$, and that $\cR$ is ergodic. Then the representation $L^{2}(\cR/\cS)$ has a nonzero $[\cR]$-invariant vector if and only if there is a $\cS$-invariant measurable set $E\subseteq X$ with $\mu(E)>0$ and so that
\[[\cR\big|_{E}:\cS\big|_{E}]<\infty.\]

\end{prop}

\begin{proof}
First suppose that there is a $\cS$-invariant measurable set $E\subseteq X$ of positive measure which satisfies
\[[\cR\big|_{E}:\cS\big|_{E}]<\infty.\]
Let $n=[\cR\big|_{E}:\cS\big|_{E}]<\infty.$
For $x\in X$, let $([x]_{\cR}\cap E)/\cS$ be the space of $\cS$-equivalence classes in $[x]_{\cR}\cap E$. Since $E$ is $\cS$-invariant, this is a well-defined subset of $[x]_{\cR}/\cS$. Since $\cR$ is ergodic,  we have that
\[|([x]_{\cR}\cap E)/\cS|=n\]
for almost every $x\in X$. Define $\xi\in L^{2}(\cR/\cS)$ by
\[\xi(x,c)=\frac{1}{\sqrt{n}}1_{([x]_{\cR}\cap E)/\cS}(c)\mbox{ for all $x\in X$.}\]
Then $\|\xi_{x}\|_{2}=1$ for almost every $x\in X$. Given $\gamma\in [\cR]$, we have
\[\xi_{x}([\gamma(x)]_{\cS})=\frac{1}{\sqrt{n}}1_{E}(\gamma(x))\]
for every $x\in X$. Thus $\xi$ is indeed a measurable vector field. It is direct to see that $\xi$ is $\cR$-invariant and not zero.

Conversely, suppose that $\xi\in L^{2}(\cR/\cS)$ is a nonzero $\cR$-invariant vector. Without loss of generality we may, and will, assume that $\|\xi\|_{2}=1$. Write $\xi=(\xi_{x})_{x}\in \Meas(\ell^{2}([x]_{\cR})/\cS)$. Choose a countable subgroup $\Delta\leq [\cR]$ so that $\Delta x=[x]_{\cR}$ for almost every $x\in X$. We may, and will, assume that there is a countable subgroup $\Lambda\leq \Delta$ so that $\Lambda x=[x]_{\cS}$ for all almost every $x\in X$. Saying that $\xi$ is $\cR$-invariant means that  for almost every $x\in X$ we have $\xi_{\gamma(x)}=\xi_{x}$ for every $\gamma\in \Delta$.  In particular, $\|\xi_{x}\|_{2}$ is a $\cR$-invariant function and thus essentially constant. We may thus find a conull subset $X_{0}$ of $X$ so that the following conditions hold:
\begin{itemize}
    \item $\|\xi_{x}\|_{2}=1$ for all $x\in X_{0}$,
    \item $\xi_{\gamma(x)}=\xi_{x}$ for all $x\in X_{0},\gamma\in\Delta$,
    \item $\Delta x=[x]_{\cR}$, and $\Lambda x=[x]_{\cS}$ for all $x\in X_{0}$,
    \item $\Delta x\subseteq X_{0}$, for all $x\in X_{0}$.
\end{itemize}

For $\varepsilon>0$, let $E_{\varepsilon}=\{x\in X_{0}:|\xi_{x}([x]_{\cS})|\geq \varepsilon\}$. Since $[\psi(x)]=[x]_{\cS}$ and $\xi_{\psi(x)}=\xi_{x}$ for all $\psi\in \Lambda$ and $x\in X_{0}$,   we see that $E_{\varepsilon}$ is $\cS$-invariant. We claim that $\mu(E_{\varepsilon})>0$ for some $\varepsilon>0$. Indeed, if this were false, then we would have that $\xi_{x}([x]_{\cS})=0$ for almost every $x\in X_{0}$. But then, by countability of $\Delta$, for almost every $x\in X_{0}$ we would have
\[\xi_{x}([\gamma(x)]_{\cS})=\xi_{\gamma(x)}([\gamma(x)]_{\cS})=0\]
for all $\gamma\in\Delta$. Since $\Delta x=[x]_{\cR}$ for all $x\in X_{0}$, this would force $\xi_{x}=0$ for almost every $x\in X$, contradicting our assumption that $\|\xi\|_{2}=1$.

So we may fix an $\varepsilon>0$ with $\mu(E_{\varepsilon})>0$. Suppose that $n\in \N$, that $x\in E_{\varepsilon}$ and that $\phi_{1},\cdots,\phi_{n}\colon E_{\varepsilon}\to E_{\varepsilon}$ are measurable functions satisfying
\begin{itemize}
    \item $\phi_{j}(x)\in [x]_{\cR}$ for all $j=1,\cdots,n$, and almost every $x\in E_{\varepsilon}$,
\item $[\phi_{j}(x)]_{\cS}\ne [\phi_{k}(x)]_{\cS}$ for all $1\leq j<k\leq n$.
\end{itemize} Then, for almost every $x\in E_{\varepsilon}$,
\[1=\|\xi_{x}\|_{2}^{2}\geq \sum_{j=1}^{n}|\xi_{x}([\phi_{j}(x)]_{\cS})|^{2}=\sum_{j=1}^{n}|\xi_{\phi_{j}(x)}([\phi_{j}(x)]_{\cS})|^{2}\geq n\varepsilon^{2},\]
where in the last step we use invariance of $\xi$ and the fact that $\phi_{j}(x)\in [x]_{\cR}$ for all $j=1,\cdots,n$ and almost every $x\in E_{\varepsilon}$. So we have $n\leq \frac{1}{\varepsilon^{2}}$. This proves that
\[[\cR\big|_{E_{\varepsilon}}:\cS\big|_{E_{\varepsilon}}]\leq \frac{1}{\varepsilon^{2}}.\]

\end{proof}

We remark that Proposition \ref{prop:translation of invariant vectors} is a natural analogue of the fact that if $\Lambda\leq \Gamma$ is an inclusion of discrete groups, then $\ell^{2}(\Gamma/\Lambda)$ has a $\Gamma$-invariant vector if and only if $[\Gamma:\Lambda]<\infty$. If one uses the framework of Jones' basic construction discussed in the introduction of \cite{AFH} (see e.g. the discussion preceding Theorem 1.3 of \cite{AFH}, or \cite[Section 6.3]{HayesCoAmen}), then one can give an alternative proof of Proposition  \ref{prop:translation of invariant vectors}. Indeed, with notation as in the introduction of \cite{AFH} the hypothesis of Proposition \ref{prop:translation of invariant vectors} imply that there is an $L(\cR)$-central vector in $L^{2}$ of the basic construction of $L(\cS)\leq L(\cR)$. Proposition  \ref{prop:translation of invariant vectors} then
follows from \cite[Lemma 2.3]{PimnserPopa} and the intertwining lemma of Popa \cite[Section 2]{PopaStrongRigidity} (see e.g. \cite[Lemma 2.8]{Bleary}, \cite[Lemma 2.4]{CAYBraid},\cite[Lemma 1.4]{IoanaHT}).

\begin{thm}\label{thm: the main result}
 Let $\cS\leq \cR$ be a coamenable inclusion of discrete, probability measure-preserving, equivalence relations on $(X,\mu)$. If $\cR$ is strongly ergodic, then there is an $\cS$-invariant, positive measure $E\subseteq X$ so that $\cS|_{E}$ is strongly ergodic. 
\end{thm}

\begin{proof}
Fix a free ultrafilter $\omega\in\beta\N\setminus\N$. 
By \cite[Proposition 6.1]{Dye} we may find a subrelation $\widehat{\cS}\leq \cR$ with $\cS\subseteq \widehat{\cS}$, and 
\[L(\widehat{\cS})=\Fix_{\cS}(L^{\infty}(X)^{\omega})'\cap L(\cR).\]
Since $\cS\leq \cR$ is coamenable, it follows by \cite[Theorem 3.2.4]{PopaCorr} (see also \cite[Theorem 3.5 (x)]{HayesCoAmen}) that $\widehat{\cS}\leq \cR$ is coamenable. Hence, by Proposition \ref{prop: coamenable weak containment characterize}, it follows that $L^{2}(X)\preceq L^{2}(\cR/\widehat{\cS})$ as representations of $C^{*}(\cR)$. 
However, the definition of $\widehat{\cS}$ forces that 
\[L(\widehat{\cS})=\Fix_{\widehat{\cS}}(L^{\infty}(X)^{\omega})'\cap L(\cR).\]
Thus, by Lemma \ref{lem: analogue of BMO}, it follows that $L^{2}(\cR/\widehat{\cS})\preceq L^{2}(X)$ as representations of $C^{*}(\cR)$. Thus $L^{2}(X)$ and $L^{2}(\cR/\widehat{\cS})$ are weakly equivalent representations of $C^{*}(\cR)$. Since $\cR$ is strongly ergodic, Corollary \ref{cor: weak containment strong ergodic} and  tell us that $L^{2}(\cR/\widehat{\cS})$ has an $[\cR]$-invariant vector. It now follows by Proposition \ref{prop:translation of invariant vectors} that there is a positive measure, $\widehat{\cS}$-invariant set $E$ so that $[\cR|_{E}:\widehat{\cS}|_{E}]<+\infty$. In particular, the fact that $\cR$ is strongly ergodic forces $\widehat{\cS}|_{E}$ to be strongly ergodic. We now use this to prove that $\cS|_{E}$ is strongly ergodic.

Indeed, suppose that $f\in L^{\infty}(E)^{\omega}$ and $\alpha_{\sigma}(f)=f$ for every $\sigma\in [\cS|_{E}]$. Since $E$ is $\widehat{\cS}$-invariant, and $\cS\leq \widehat{\cS}$, we then have that $f\in \Fix_{\cS}(L^{\infty}(X)^{\omega})$. But then the definition of $\widehat{\cS}$ forces $f$ to commute with $L(\widehat{\cS})$. This implies that $\alpha_{\sigma}(f)=f$ for every $\sigma\in [\widehat{\cS}]$. Again using the $\widehat{\cS}$-invariance of $E$ we conclude that $\alpha_{\sigma}(f)=f$ for every $\sigma\in [\widehat{\cS}|_{E}]$. Since $\widehat{\cS}|_{E}$ is strongly ergodic, we conclude that $f\in \C1$. This shows that $\cS|_{E}$ is strongly ergodic.

\end{proof}

\begin{cor}
 Suppose that $\cS\leq \cR$ is an everywhere coamenable inclusion of probability measure-preserving equivalence relations on $(X,\mu)$. If $\cR$ is strongly ergodic, then the ergodic decomposition of $\cS$ is atomic, and each ergodic component of $\cS$ is strongly ergodic. 
\end{cor}

\begin{proof}
Apply Zorn's lemma to find a maximal collection $(E_{i})_{i\in I}$ of pairwise disjoint, $\cS$-invariant, positive measure sets such that $\cS|_{E_{i}}$ is strongly ergodic. Since each $E_{i}$ is positive measure, and $\mu$ is a probability measure, we necessarily have that $I$ is countable. Set $E=\bigcup_{i\in I}E_{i}$. Suppose, for contradiction, that $\mu(E)<1$. Then since $\cS\leq \cR$ is everywhere coamenable, we have that $\cS|_{E^{c}}\leq \cR|_{E^{c}}$ is a coamenable inclusion. Thus Theorem \ref{thm: the main result} implies that there is a positive measure, $\cS$-invariant set $F\subseteq E^{c}$ with $\cS|_{F}$ strongly ergodic. This contradicts maximality of $(E_{i})_{i\in I}$. Hence, $\mu(E)=1$. So $(E_{i})_{i\in I}$ forms a partition modulo null sets into $\cS$-invariant positive measures sets with $\cS|_{E_{i}}$ being strongly ergodic for all $i\in I$.

\end{proof}


\begin{thebibliography}{10}

\bibitem{AFH}
M.~Abert, M.~Fraczyk, and B.~Hayes.
\newblock Co-spectral radius for countable equivalence relations, 2023.

\bibitem{BMOFull}
J.~Bannon, A.~Marrakchi, and N.~Ozawa.
\newblock Full factors and co-amenable inclusions.
\newblock {\em Comm. Math. Phys.}, 378(2):1107--1121, 2020.

\bibitem{BekkaOAsuperrigid}
B.~Bekka.
\newblock Operator-algebraic superridigity for {${\rm SL}_n(\Bbb Z)$}, {$n\geq
  3$}.
\newblock {\em Invent. Math.}, 169(2):401--425, 2007.

\bibitem{BHV}
B.~Bekka, P.~de~la Harpe, and A.~Valette.
\newblock {\em Kazhdan's property ({T})}, volume~11 of {\em New Mathematical
  Monographs}.
\newblock Cambridge University Press, Cambridge, 2008.

\bibitem{BO}
N.~P. Brown and N.~Ozawa.
\newblock {\em {$C^*$}-algebras and finite-dimensional approximations},
  volume~88 of {\em Graduate Studies in Mathematics}.
\newblock American Mathematical Society, Providence, RI, 2008.

\bibitem{SolidErg}
I.~Chifan and A.~Ioana.
\newblock Ergodic subequivalence relations induced by a {B}ernoulli action.
\newblock {\em Geom. Funct. Anal.}, 20(1):53--67, 2010.

\bibitem{CAYBraid}
I.~Chifan, A.~Ioana, and Y.~Kida.
\newblock {$W^*$}-superrigidity for arbitrary actions of central quotients of
  braid groups.
\newblock {\em Math. Ann.}, 361(3-4):563--582, 2015.

\bibitem{ConnesAP}
A.~Connes.
\newblock Almost periodic states and factors of type {${\rm III}\sb{1}$}.
\newblock {\em J. Functional Analysis}, 16:415--445, 1974.

\bibitem{ConnesSE}
A.~Connes.
\newblock Outer conjugacy classes of automorphisms of factors.
\newblock {\em Ann. Sci. \'Ecole Norm. Sup. (4)}, 8(3):383--419, 1975.

\bibitem{CFW}
A.~Connes, J.~Feldman, and B.~Weiss.
\newblock An amenable equivalence relation is generated by a single
  transformation.
\newblock {\em Ergodic Theory Dynamical Systems}, 1(4):431--450 (1982), 1981.

\bibitem{ConnesJonesCartan}
A.~Connes and V.~Jones.
\newblock A {${\rm II}_{1}$} factor with two nonconjugate {C}artan subalgebras.
\newblock {\em Bull. Amer. Math. Soc. (N.S.)}, 6(2):211--212, 1982.

\bibitem{CWPropT}
A.~Connes and B.~Weiss.
\newblock Property {${\rm T}$} and asymptotically invariant sequences.
\newblock {\em Israel J. Math.}, 37(3):209--210, 1980.

\bibitem{Conway}
J.~B. Conway.
\newblock {\em A course in functional analysis}, volume~96 of {\em Graduate
  Texts in Mathematics}.
\newblock Springer-Verlag, New York, second edition, 1990.

\bibitem{ConwayOT}
J.~B. Conway.
\newblock {\em A course in operator theory}, volume~21 of {\em Graduate Studies
  in Mathematics}.
\newblock American Mathematical Society, Providence, RI, 2000.

\bibitem{csoka2025quantitativeindistinguishabilitysparsedense}
E.~Csóka, P.~Mester, and G.~Pete.
\newblock Quantitative indistinguishability and sparse and dense clusters in
  factor of iid percolations, 2025.

\bibitem{DixmierC*}
J.~Dixmier.
\newblock {\em {$C\sp*$}-algebras}, volume Vol. 15 of {\em North-Holland
  Mathematical Library}.
\newblock North-Holland Publishing Co., Amsterdam-New York-Oxford, 1977.
\newblock Translated from the French by Francis Jellett.

\bibitem{DixmierW*}
J.~Dixmier.
\newblock {\em von {N}eumann algebras}, volume~27 of {\em North-Holland
  Mathematical Library}.
\newblock North-Holland Publishing Co., Amsterdam-New York, french edition,
  1981.
\newblock With a preface by E. C. Lance.

\bibitem{Dye}
H.~A. Dye.
\newblock On groups of measure preserving transformations $\textrm{I}$.
\newblock {\em Amer. Journ. Math.}, 81(1):119--159, 1959.

\bibitem{ElekLip}
G.~Elek and G.~Lippner.
\newblock Sofic equivalence relations.
\newblock {\em J. Funct. Anal.}, 258(5):1692--1708, 2010.

\bibitem{FelMooreI}
J.~Feldman and C.~C. Moore.
\newblock Ergodic equivalence relations, cohomology, and von {N}eumann
  algebras. {I}.
\newblock {\em Trans. Amer. Math. Soc.}, 234(2):289--324, 1977.

\bibitem{FelMoore}
J.~Feldman and C.~C. Moore.
\newblock Ergodic equivalence relations, cohomology, and von {N}eumann
  algebras. {II}.
\newblock {\em Trans. Amer. Math. Soc.}, 234(2):325--359, 1977.

\bibitem{floresharbourgroupoid}
F.~Flores and J.~Harbour.
\newblock Discrete measured groupoid von neumann algebras via the gaussian
  deformation, 2025.

\bibitem{Gaboriau-Lyons}
D.~Gaboriau and R.~Lyons.
\newblock A measurable-group-theoretic solution to von {N}eumann's problem.
\newblock {\em Invent. Math.}, 177(3):533--540, 2009.

\bibitem{giritlioglu2026translationactionsnonunimodulargroups}
F.~E. Giritlioglu.
\newblock Translation actions on non-unimodular groups and strong ergodicity,
  2026.

\bibitem{HayesCoAmen}
B.~Hayes.
\newblock Coamenability and cospectral radius for orbit equivalence relations,
  2024.

\bibitem{HVTypeIIICartan}
C.~Houdayer and S.~Vaes.
\newblock Type {III} factors with unique {C}artan decomposition.
\newblock {\em J. Math. Pures Appl. (9)}, 100(4):564--590, 2013.

\bibitem{IoanaHT}
A.~Ioana.
\newblock Uniqueness of the group measure space decomposition for {P}opa's
  {$\mathscr{HT}$} factors.
\newblock {\em Geom. Funct. Anal.}, 22(3):699--732, 2012.

\bibitem{JonesSchmidt}
V.~F.~R. Jones and K.~Schmidt.
\newblock Asymptotically invariant sequences and approximate finiteness.
\newblock {\em Amer. J. Math.}, 109(1):91--114, 1987.

\bibitem{KadisonRingroseII}
R.~V. Kadison and J.~R. Ringrose.
\newblock {\em Fundamentals of the theory of operator algebras. {V}ol. {II}},
  volume~16 of {\em Graduate Studies in Mathematics}.
\newblock American Mathematical Society, Providence, RI, 1997.
\newblock Advanced theory, Corrected reprint of the 1986 original.

\bibitem{Bleary}
B.~Leary.
\newblock Maximal amenability with asymptotic orthogonality in amalgamated free
  products.
\newblock {\em J. Operator Theory}, 86(1):17--29, 2021.

\bibitem{OrnWeiss}
D.~S. Ornstein and B.~Weiss.
\newblock Entropy and isomorphism theorems for actions of amenable groups.
\newblock {\em J. Analyse Math.}, 48:1--141, 1987.

\bibitem{PimnserPopa}
M.~Pimsner and S.~Popa.
\newblock Entropy and index for subfactors.
\newblock {\em Ann. Sci. \'{E}cole Norm. Sup. (4)}, 19(1):57--106, 1986.

\bibitem{PopaCorr}
S.~Popa.
\newblock Correspondences.
\newblock {\em INCREST preprint, unpublished.}, 1986.

\bibitem{PopaStrongRigidity}
S.~Popa.
\newblock Strong rigidity of {$\rm II_1$} factors arising from malleable
  actions of {$w$}-rigid groups. {I}.
\newblock {\em Invent. Math.}, 165(2):369--408, 2006.

\bibitem{SauerBetti}
R.~Sauer.
\newblock {$L^2$}-{B}etti numbers of discrete measured groupoids.
\newblock {\em Internat. J. Algebra Comput.}, 15(5-6):1169--1188, 2005.

\bibitem{SchmidtCohom}
K.~Schmidt.
\newblock Asymptotically invariant sequences and an action of {${\rm
  SL}(2,\,{\bf Z})$}\ on the {$2$}-sphere.
\newblock {\em Israel J. Math.}, 37(3):193--208, 1980.

\bibitem{SchmidtSpectralGap}
K.~Schmidt.
\newblock Amenability, {K}azhdan's property {$T$}, strong ergodicity and
  invariant means for ergodic group-actions.
\newblock {\em Ergodic Theory Dynamical Systems}, 1(2):223--236, 1981.

\bibitem{Taka}
M.~Takesaki.
\newblock {\em Theory of operator algebras. {I}}, volume 124 of {\em
  Encyclopaedia of Mathematical Sciences}.
\newblock Springer-Verlag, Berlin, 2002.
\newblock Reprint of the first (1979) edition, Operator Algebras and
  Non-commutative Geometry, 5.

\bibitem{AW25}
A.~Wu.
\newblock {\em von Neumann Orbit Equivalence}.
\newblock PhD thesis, University of Virginia, 2025.

\bibitem{ZimmerAmen2}
R.~J. Zimmer.
\newblock Cocycles and the structure of ergodic group actions.
\newblock {\em Israel J. Math.}, 26(3-4):214--220, 1977.

\bibitem{ZimmerAMen1}
R.~J. Zimmer.
\newblock Hyperfinite factors and amenable ergodic actions.
\newblock In {\em Group actions in ergodic theory, geometry, and
  topology---selected papers}, pages 152--160. Univ. Chicago Press, Chicago,
  IL, 2020.
\newblock Reprint of [0470692].

\end{thebibliography}

\end{document}